\documentclass[11pt,twoside]{article}

\usepackage[left=2.3cm,right=2.3cm,top=2.5cm,bottom=2.5cm]{geometry}
\usepackage{natbib} \setcitestyle{round}
\usepackage{amsmath, amssymb, amsthm, graphicx,refcount}
\usepackage[colorlinks,
	linkcolor=blue,
	citecolor=blue,
	urlcolor=magenta,
	linktocpage,
	plainpages=false,breaklinks=true]{hyperref}
\usepackage[dvipsnames]{xcolor}
\usepackage[utf8x]{inputenc}
\usepackage{breakcites}

\usepackage{fancyhdr}
\date{}

\title{An $M^*$ Proxy for Sparse Recovery Performance}
\usepackage{times}
\usepackage{algorithm}
\usepackage{algorithmic}
\renewcommand{\algorithmicrequire}{\textbf{Input:}} 
\renewcommand{\algorithmicensure}{\textbf{Output:}}

\newcommand{\N}{\mathbb{N}}
\makeatletter
\newcommand{\seqref}[1]{\textup{\tagform@split{\getrefnumber{#1}}}}
\newcommand\tagform@split[1]{%
  \begingroup
  \m@th\normalfont(\ignorespaces #1\unskip\@@italiccorr)%
  \endgroup
}
\makeatother

\usepackage {tikz}
\usepackage{pgfkeys}
\usetikzlibrary{shapes,arrows}
\usepackage{pgfplots}
\pgfplotsset{compat=1.13}
\pgfplotsset{plotOptions/.style={%
		label style={font=\scriptsize},
		legend style={font=\scriptsize},
		tick label style={font=\scriptsize},
		solid,
		very thick
	}}
\definecolor{colorP1}{RGB}{55,126,184}  
\definecolor{colorP2}{RGB}{228,26,28}  
\definecolor{colorP3}{RGB}{152,78,163} 
\definecolor{colorP4}{RGB}{77,175,74}  
\definecolor{colorP5}{RGB}{250, 150, 10} 
\definecolor{colorP6}{cmyk}{0,0.5,1,0}

\author{Mathieu Barr\'e\footnote{
	INRIA - D\'epartement d'informatique de l'ENS, CNRS, 
	PSL University, Paris, France. Email: \emph{mathieu.barre@inria.fr}.},\ \ Alexandre d'Aspremont\footnote{D\'epartement d'informatique de l'ENS, CNRS, PSL University, Paris, France. Email: \emph{aspremon@ens.fr}.} \\[.3cm]}

\begin{document}
		
\definecolor{ddarkbrown}{rgb}{0.5,0.2,0.05} \definecolor{bbluegray}{rgb}{0.05,0,0.5}

\newtheorem{theorem}{Theorem}[section]
\newtheorem{proposition}[theorem]{Proposition}
\newtheorem{definition}[theorem]{Definition}
\newtheorem{lemma}[theorem]{Lemma}
\newtheorem{corollary}[theorem]{Corollary}
\newtheorem{remark}[theorem]{Remark}

\newcommand{\BEAS}{\begin{eqnarray*}}
\newcommand{\EEAS}{\end{eqnarray*}}
\newcommand{\BEA}{\begin{eqnarray}}
\newcommand{\EEA}{\end{eqnarray}}
\newcommand{\BEQ}{\begin{equation}}
\newcommand{\EEQ}{\end{equation}}
\newcommand{\BIT}{\begin{itemize}}
\newcommand{\EIT}{\end{itemize}}
\newcommand{\BNUM}{\begin{enumerate}}
\newcommand{\ENUM}{\end{enumerate}}

\newcommand{\BA}{\begin{array}}
\newcommand{\EA}{\end{array}}

\newcommand{\refp}[1]{(\ref{#1})}

\newcommand{\cf}{{\it cf.}}
\newcommand{\eg}{{\it e.g.}}
\newcommand{\ie}{{\it i.e.}}
\newcommand{\etc}{{\it etc.}}
\newcommand{\ones}{\mathbf 1}

\newcommand{\reals}{{\mathbb R}}
\newcommand{\sreals}{\scriptsize{\mbox{\bf R}}}
\newcommand{\integers}{{\mbox{\bf Z}}}
\newcommand{\eqbydef}{\mathrel{\stackrel{\Delta}{=}}}
\newcommand{\complex}{{\mbox{\bf C}}}
\newcommand{\symm}{{\mbox{\bf S}}}  

\newcommand{\Span}{\mbox{\textrm{span}}}
\newcommand{\Range}{\mbox{\textrm{range}}}
\newcommand{\nullspace}{{\mathcal N}}
\newcommand{\range}{{\mathcal R}}
\newcommand{\diam}{\mathop{\bf radius}}
\newcommand{\sphere}{{\mathbb S}}
\newcommand{\Nullspace}{\mbox{\textrm{nullspace}}}
\newcommand{\Rank}{\mathop{\bf Rank}}
\newcommand{\NumRank}{\mathop{\bf NumRank}}
\newcommand{\NumCard}{\mathop{\bf NumCard}}
\newcommand{\Card}{\mathop{\bf Card}}
\newcommand{\Tr}{\mathop{\bf Tr}}
\newcommand{\diag}{\mathop{\bf diag}}
\newcommand{\lambdamax}{{\lambda_{\rm max}}}
\newcommand{\lambdamin}{\lambda_{\rm min}}
\newcommand{\idm}{\mathbf{I}}

\newcommand{\Expect}{\textstyle{\bf E}}
\newcommand{\Median}{\textstyle\mathop{\bf M}}
\newcommand{\Prob}{\mathop{\bf Prob}}
\newcommand{\erf}{\mathop{\bf erf}}

\newcommand{\Co}{{\mathop {\bf Co}}}
\newcommand{\co}{{\mathop {\bf Co}}}
\newcommand{\Var}{\mathop{\bf var{}}}
\newcommand{\dist}{\mathop{\bf dist{}}}
\newcommand{\Ltwo}{{\bf L}_2}
\newcommand{\QED}{~~\rule[-1pt]{6pt}{6pt}}\def\qed{\QED}
\newcommand{\approxleq}{\mathrel{\smash{\makebox[0pt][l]{\raisebox{-3.4pt}{\small$\sim$}}}{\raisebox{1.1pt}{$<$}}}}
\newcommand{\argmin}{\mathop{\rm argmin}}
\newcommand{\epi}{\mathop{\bf epi}}
\newcommand{\var}{\mathop{\bf var}}

\newcommand{\vol}{\mathop{\bf vol}}
\newcommand{\Vol}{\mathop{\bf vol}}

\newcommand{\dom}{\mathop{\bf dom}}
\newcommand{\aff}{\mathop{\bf aff}}
\newcommand{\cl}{\mathop{\bf cl}}
\newcommand{\Angle}{\mathop{\bf angle}}
\newcommand{\intr}{\mathop{\bf int}}
\newcommand{\relint}{\mathop{\bf rel int}}
\newcommand{\bd}{\mathop{\bf bd}}
\newcommand{\vect}{\mathop{\bf vec}}
\newcommand{\dsp}{\displaystyle}
\newcommand{\foequal}{\simeq}
\newcommand{\VOL}{{\mbox{\bf vol}}}
\newcommand{\argmax}{\mathop{\rm argmax}}
\newcommand{\xopt}{x^{\rm opt}}

\newcommand{\Xb}{{\mbox{\bf X}}}
\newcommand{\xst}{x^\star}
\newcommand{\varphist}{\varphi^\star}
\newcommand{\lambdast}{\lambda^\star}
\newcommand{\Zst}{Z^\star}
\newcommand{\fstar}{f^\star}
\newcommand{\xstar}{x^\star}
\newcommand{\xc}{x^\star}
\newcommand{\lambdac}{\lambda^\star}
\newcommand{\lambdaopt}{\lambda^{\rm opt}}

\newcommand{\geqK}{\mathrel{\succeq_K}}
\newcommand{\gK}{\mathrel{\succ_K}}
\newcommand{\leqK}{\mathrel{\preceq_K}}
\newcommand{\lK}{\mathrel{\prec_K}}
\newcommand{\geqKst}{\mathrel{\succeq_{K^*}}}
\newcommand{\gKst}{\mathrel{\succ_{K^*}}}
\newcommand{\leqKst}{\mathrel{\preceq_{K^*}}}
\newcommand{\lKst}{\mathrel{\prec_{K^*}}}
\newcommand{\geqL}{\mathrel{\succeq_L}}
\newcommand{\gL}{\mathrel{\succ_L}}
\newcommand{\leqL}{\mathrel{\preceq_L}}
\newcommand{\lL}{\mathrel{\prec_L}}
\newcommand{\geqLst}{\mathrel{\succeq_{L^*}}}
\newcommand{\gLst}{\mathrel{\succ_{L^*}}}
\newcommand{\leqLst}{\mathrel{\preceq_{L^*}}}
\newcommand{\lLst}{\mathrel{\prec_{L^*}}}

\newcommand{\realsp}{\mathbf{R}_+^n}
\newcommand{\intrealsp}{\int_{\mathbf{R}_+^n}}

\maketitle

\begin{abstract}
This paper provides a new tractable lower bound for the sparse recovery threshold of sensing matrices. This lower bound is used as a proxy to quantify the quality of sensing matrices in two different applications. First, it serves as regularization for the classical dictionary learning problem in order to learn dictionaries with better generalisation properties on unseen data. Then, the proxy is used to design sampling schemes for MRI acquisition that exhibit high reconstruction performances.
\end{abstract}

\section{Introduction}\label{s:intro}

\paragraph{Compressed Sensing} In the classical compressed sensing setting, we let $A\in\reals^{m \times n}$ be a full rank matrix, we are given $m$ observations $Ax_0$ of a signal $x_0\in\reals^n$, and we seek to decode it by solving
\BEQ\label{eq:l0-dec}
\BA{ll}
\mbox{minimize} & \Card(x)\\
\mbox{subject to} & Ax=Ax_0,
\EA
\EEQ
in the variable $x\in\reals^n$, with $\Card(x)$ being the number of nonzero coefficients of $x$. Problem~\eqref{eq:l0-dec} is combinatorially hard, but under certain conditions on the matrix $A$ (see e.g. \citep{Cand05,Dono05,Kash07,Cohe06}), we can reconstruct the signal by solving instead
\BEQ\label{eq:l1-dec}
\BA{ll}
\mbox{minimize} & \|x\|_1\\
\mbox{subject to} & Ax=Ax_0,
\EA
\EEQ
which is a convex problem in the variable $x\in\reals^n$. Given a sensing matrix $A$, \citep{Cand05,Dono05} showed that there is a {\em recovery threshold} $k$ such that solving problem~\eqref{eq:l1-dec} will always recover the solution to~\eqref{eq:l0-dec} provided this solution has at most $k$ nonzero coefficients. 

\paragraph{Recovery Threshold} The compressed sensing results detailed above show that to each sampling matrix $A$ corresponds a {\em recovery threshold} $k(A)$ such that solving problem~\eqref{eq:l1-dec} will always recover the solution to~\eqref{eq:l0-dec} provided this solution has less than $k(A)$ nonzero coefficients, and recovery will usually fail for signals with more than $k(A)$ nonzero coefficients. Our main idea in this work is to {\em use this recovery threshold to measure dictionary performance.} However, while many results bound the threshold $k(A)$ with high probability for certain classes of random matrices $A$, computing this threshold $k(A)$ given $A$ is a hard problem \citep{Band13,Wang16,Weed17}. Figure~\ref{fig:k-thresh} shows experimental recovery threshold for a poor sensing matrix $(A)_{ij} = \cos^2(ij)$ with $i,j \in [1,100]\times[1,200]$ and a Gaussian matrix.
\begin{figure}[!ht]
    \centering
    \includegraphics[width=0.38\textwidth]{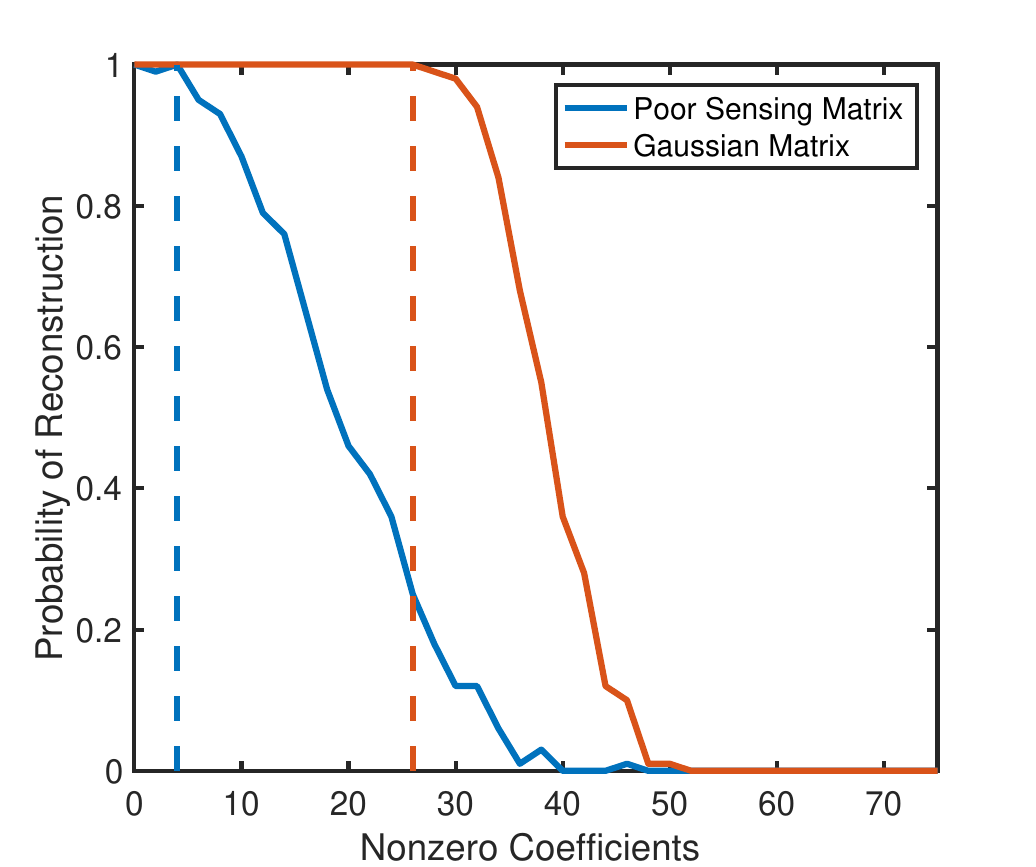}
    \caption{Empirical probability of reconstructing a random signal of dimension $200$ given $100$ linear observations by solving the LP in~\eqref{eq:l1-dec}, versus number of nonzero coefficients in the signal, for a poor sensing matrix and an i.i.d. Gaussian matrix. Dashed lines represent the empirical recovery thresholds $k(A)$ for both matrices.}
    \label{fig:k-thresh}
    
\end{figure}

Indeed, several relaxations have been derived to approximate the recovery threshold \citep{dAsp08a,Judi08,dAsp12a}, but these approximations typically only certify recovery up to signal size $\sqrt{k}$ when the optimal threshold is $k$. There is substantial evidence that this is the best that can be achieved in polynomial time, in the regimes that are relevant for compressed sensing. It can be shown for example that certifying recovery using the restricted isometry property is equivalent to solving a sparse PCA problem, which is hard in a broad sense by reduction of the planted-clique problem \citep{Bert13}. 

\paragraph{Proxies for the Recovery Threshold} In contrast with the negative results above, a simple result in \citet{Kash07} shows that the sparse recovery threshold satisfies $k(A)\geq S(A)^{-2}$, where $S(A)$ is the radius of the intersection of the unit $\ell_1$-ball with the nullspace of the sampling matrix~$A$. Directly approximating the radius of a convex polytope is of course hard \citep{Freu85,Lova92}, but the low-$M^*$ bound in e.g. \citep{Pajo86} shows that we can accurately quantify the performance of a slightly enlarged sampling matrix 
${\begin{bmatrix} A\\  G \end{bmatrix},}$
with high probability, where $G$ is a matrix containing a few additional Gaussian samples. This produces a bound on the radius (and hence on the recovery threshold) that depends only on the size of $G$ and on the $M^*$ of a section of the $\ell_1$-ball by the nullspace of the sampling matrix~$A$ (all these quantities will be precisely defined below). In what follows, we will use 
\[
k\left({\begin{bmatrix} A\\  G \end{bmatrix}}\right)
\]
as a proxy to $k(A)$. we will then use the $M^*$ bound to control this threshold and regularize dictionary learning problems. 

\paragraph{Compressed Sensing MRI} Structured acquisition in compressed sensing seeks to design dictionaries to maximize signal recovery performance while satisfying design constraints in the sampling procedure (e.g. magnetic resonance imaging hardware constraints, see~\citep{Boye16,Boye17} for a more complete discussion). In compressed sensing MRI~\citep{Lust08} for example, this means ensuring imaging samples follow a continuous path in Fourier space. More recently, structured acquisition procedures in e.g. \citep{Chau14,Boye16,Boye17} use a sampling approach based on the results of~\citep{Lust08,Cand11c}. Here, we aim at using our proxy for the recovery threshold as a tractable metric for MRI reconstruction performance. 

\paragraph{Dictionary Learning} Similarly, dictionary learning seeks to decompose and compress signals on a few atoms using a dictionary learned from the data set, instead of a predefined one formed by e.g. wavelet transforms~\citep{Mall99} or random vectors. This learning approach has significantly improved state-of-the-art performance on various signal processing tasks such as image denoising \citep{Elad06} or inpainting \citep{Mair09} for example. From a computational point of view, dictionary learning is an inherently nonconvex problem and the references above use alternating minimization to find good solutions. Furthermore, in all these cases, the dictionary learning problem is not regularized. {\em Classical methods thus learn dictionaries without proper regularization, which potentially hurts generalization performance.} One of our objective is thus to our proxy for the recovery threshold of dictionaries as a regularization term that directly improves the dictionary's performance on new data samples.

\paragraph{Contributions} Our contribution is twofold. In the spirit of smoothed analysis~\citep{Spie04}, we first show how to use the recovery threshold of a slightly enlarged dictionary matrix as a proxy for the recovery threshold of the original dictionary $A$. We then show that this threshold is controlled by the $M^*$ of a section of the $\ell_1$ ball by the nullspace of $A$. 

Then we demonstrate empirically that this $M^*$ based proxy of the recovery threshold is also a good proxy for the reconstruction (or generalization) performances of a dictionary. We do so on two applications, dictionary learning and compressed sensing MRI. 

The paper is organized as follows. In Section~\ref{s:dico} we recall the low-$M^*$ bounds and its application to sparse recovery. In Section~\ref{s:algos}, we detail our $M^*$-regularized dictionary learning algorithms. We detail numerical experiments on inpainting by dictionary learning in Section~\ref{s:numres}. Finally, we provide a $M^*$ based greedy algorithm for designing MRI sampling matrices and numerical experiments in Section~\ref{s:mri}.

\paragraph{Notations}
$B_1^n$ is the unit ball for the $l_1$ norm in $\reals^n$. Given a matrix $A\in\reals^{m\times n}$ , $\nullspace(A)$ will denote depending on the context, either the nullspace of the matrix $A$ or a matrix with its columns being an orthonormal basis of the nullspace of $A$. Given a symmetric set $X \subset \reals^n$ containing $0$ , the euclidean radius of $X$, $\diam(X)$ is defined as $\underset{x\in X}{\sup}\; \|x\|_2$.\\
Given a symmetric convex set $\mathcal{K}\subset\reals^n$ with $0\in\mathcal{K}$, we denote by $\mathcal{K}^*$ its polar set and $\|\cdot\|_{\mathcal{K}}$ the unique norm that admits $\mathcal{K}$ as unit ball.\\
Throughout the paper, we use the terms sensing matrix, dictionary or sampling matrix depending on the context but they all correspond to matrices from which observations are obtained by applying it to an original signal.

\section{Low M* dictionaries}\label{s:dico}
Our starting point is the following result by \citet{Kash07}, linking recovery thresholds and the radius of a section of the unit $\ell_1$ ball. This last quantity, while hard to estimate, then provides accurate bounds on the recovery threshold $k$ of a given sampling matrix $A$.

\begin{proposition}{\citep[Th. 2.1]{Kash07}}\label{prop:card-k}
Given $m<n\in\N$ and a coding matrix $A\in\reals^{m \times n}$ with full rank, suppose that there is some $k>0$ such that
\BEQ\label{eq:diam}
S(A)\triangleq \sup_{Ax=0} \frac{\|x\|_2}{\|x\|_1} \leq \frac{1}{\sqrt{k}}
\EEQ
then 
\[\left\{\begin{array}{cc}
     x^\mathrm{LP}=x_0 & \text{ if } \Card(x_0) < k/4,  \\
     \|x_0-x^\mathrm{LP}\|_1 \leq ~4\min_{\{\Card(y)\leq k/16\}} ~\|x_0-y\|_1 & \text{ otherwise}
\end{array}\right.
\]
where $x^\mathrm{LP}$ solves the $\ell_1$-recovery problem in~\eqref{eq:l1-dec} and $x_0$ is the original signal.
\end{proposition}
This result means that the $\ell_1$-minimization problem in~\eqref{eq:l1-dec} will recover exactly all sparse signals $x_0$ satisfying $\Card(x_0)< (2S(A))^{-2}$ and that the $\ell_1$ reconstruction error for other signals will be at most four times larger than the $\ell_1$ error corresponding to the best possible approximation of $x_0$ by a signal of cardinality at most $(4S(A))^{-2}$. The quantity $S(A)=\sup_{Ax=0} {\|x\|_2}/{\|x\|_1}$ is the radius of the intersection of the unit $\ell_1$ ball with the nullspace of $A$, written
\BEQ\label{eq:K-ball}
\mathcal{K}(A)\triangleq \{x\in\reals^n: \|x\|_1 \leq 1,\, Ax=0\},
\EEQ
and thus controls the recovery threshold of the dictionary matrix $A$, i.e. the largest signal size that can provably and uniformly be recovered using the observations in $A$. In what follows, we first recall some classical results on this radius from geometric functional analysis and use them to quantify the sparse recovery thresholds of arbitrary matrices~$A$.


\subsection{Low M* Bounds}
The radius of $\mathcal{K}(A)$ can be precisely quantified for some classes of random matrices $A$, and the following result characterizes its behavior as the dimension $A$ varies, in terms of a quantity called $M^*$ defined as follows for a set $K \subset \reals^n$.
\BEQ\label{def:w}
M^*(K) \triangleq \Expect\left[\underset{x \in K}{\sup} \langle g,x \rangle \right]
\EEQ
where $g \sim \mathcal{N}(0,\idm_{n})$. The classical definition uses the uniform measure on the sphere instead of Gaussian variables. We use the Gaussian width formulation for simplicity here, but they only differ by a factor $\sqrt{n}$.


\begin{theorem}{(Low $M^*$ bound)} \label{eq:low-m*}
Given a symmetric set $K \subset \reals^n$ with $0\in K$ and an $m \times n$ matrix $A$ whose rows are independent isotropic Gaussian random vectors in $\reals^n$, the radius of a section of $K$ by the nullspace of $A$ satisfies
\[ \diam(K\cap \nullspace(A)) \leq \frac{c}{\sqrt{m}}M^*(K)\]
with probability $1-e^{-c'm}$, where $c,c'>0$ are absolute constants.
\end{theorem}
\begin{proof}
See \citep{Pajo86} or \citep[Th. 9.4.2]{Vers18} for example.
\end{proof}


The value of $M^*(K)$ is known for many convex bodies, including $l_p$ balls. In particular, $M^*(B_1^n)$ simply reduces to $\Expect[\|g\|_{\infty}]\sim\sqrt{\log n}$ asymptotically. 
This means that given $A \in \reals^{m\times n}$ a random matrix with rows being independent isotropic Gaussian vectors, Theorem~\ref{eq:low-m*} gives the bound on the radius
\[
\diam(B_1^n \cap \nullspace(A))\leq c \sqrt{\frac{\log n}{m}}
\]
with high probability, where $c$ is an absolute constant (a more precise analysis allows the $\log$ term to be replaced by $log(n/m)$). This, combined with the result of Proposition~\ref{prop:card-k}, shows that  problem~\eqref{eq:l1-dec} will recover all signals with at most $m/(4c^2 \log n)$ coefficients with high probability. Theorem~\ref{eq:low-m*} combined with Proposition~\ref{prop:card-k} is thus a fairly direct way of proving optimal bounds on the sparse recovery threshold for random matrices $A$.

\subsection{Recovery threshold of a Augmented Matrix}
The main idea here is to apply Theorem~\ref{eq:low-m*} to the intersection of the unit $\ell_1$ ball with the nullspace of $A$ defined as $\mathcal{K}(A)$ in \eqref{eq:K-ball} and get a bound on the radius of random sections of $\mathcal{K}(A)$.
Let
\[
\mathcal{K}\left({\begin{bmatrix} A\\  G \end{bmatrix}}\right)=\{x\in \reals^n:Ax=0,\,Gx=0,\,\|x\|_1\leq1\}
\]
for some $G\in\reals^{q \times n}$, we get the following corollary. 

\begin{corollary}\label{cor:mstar-A}
Let $A\in\reals^{m\times n}$ be a given matrix and $G\in\reals^{q \times n}$ be a matrix with i.i.d. Gaussian coefficients, then 
\[
\diam\left(\mathcal{K}\left({\begin{bmatrix} A\\  G \end{bmatrix}}\right)\right) \leq \frac{c}{\sqrt{q}}M^*(A)
\]
with probability $1-e^{-c'q}$, where $c,c'>0$ are absolute constants, where for simplicity, we have written 
\BEQ\label{def:mstar-A}
M^*(A)\triangleq M^*(\mathcal{K}(A))
\EEQ
the $M^*$ of the set $\mathcal{K}(A)$.
\end{corollary}
\begin{proof}
$\mathcal{K}(A)$ is symmetric and contains $0$ thus applying Theorem~\ref{eq:low-m*} to 
\[
\mathcal{K}\left({\begin{bmatrix} A\\  G \end{bmatrix}}\right) = \mathcal{K}(A) \cap \nullspace(G)
\]
yields the desired result.
\end{proof}

This last result shows that we can control the radius of $\{x\in \reals^n:Ax=0,\,Gx=0,\,\|x\|_1\leq1\}$, hence the recovery threshold of the augmented sampling matrix with high probability. We summarize the link between sparse recovery performance of $A$ and $M^*(A)$ in the following proposition, which is the main result of this section.

\begin{proposition}\label{prop:mstar-reco}
Let $A\in\reals^{m\times n}$ be a given matrix and $G\in\reals^{q \times n}$ be a matrix with i.i.d. Gaussian coefficients. Suppose $x^{LP}$ solves
\BEQ\label{eq:l0-dec-G}
\BA{ll}
\mbox{minimize} & \|x\|_1\\
\mbox{subject to} & Ax=Ax_0,\\
& Gx=Gx_0,
\EA
\EEQ
in the variable $x\in\reals^n$, then with probability $1-e^{-c'q}$, $x^{LP}=x_0$ provided
\[
\Card{x_0} \leq \frac{q}{4cM^*(A)^2}
\]
where $c,c'>0$ are absolute constants.
\end{proposition}
\begin{proof}
We simply combine the result of Corollary~\ref{cor:mstar-A} with that of Proposition~\ref{prop:card-k}.
\end{proof}

This last result means that, everything else being equal (problem dimensions $n,m$, reliability $q$, \ldots), the recovery performance of $A$ when solving problem~\eqref{eq:l0-dec-G} is controlled by $M^*(A)$. In what follows, we will use the recovery performance of problem~\eqref{eq:l0-dec-G} as a proxy for that of problem~\eqref{eq:l1-dec}. As we will see below, $M^*(A)$ is a quantity that we can both estimate and optimize.

\subsection{Estimating M*(A)} 




$M^*(A)$ can be approximated by simulation, computing
\BEQ\label{eq:l1-m*}
\BA{rl}
M^*(A)&\triangleq \Expect\left[\max_{\substack{\|x\|_1 \leq 1 \\ Ax = 0}} x^Tg \right] \\
\EA
\EEQ
Estimating $M^*(A)$ thus means solving one linear program per sample. 

Also, $M^*(A)$ is bounded above by $O(\sqrt{\log(n)})$ and below by $O(\frac{1}{\sqrt{m}})$ w.h.p.
It gives us a target relative precision for our estimate of $M^*$ in $o(\frac{1}{m})$ and produces a recipe for a randomized polynomial time algorithm for estimating $M^*(A)$. In fact, following \citep{Bour88,Gian97,Gian05}, we have the following bound. 
\begin{proposition}
Given a matrix $A\in \reals^{m \times n}$, $0<\delta,\beta<1$, we draw $N$ points $g_i \sim \mathcal{N}(0,\idm_n)$ with
$N=\lceil\frac{c \log(2/\beta)}{\delta^2}\rceil+1$
where $c$ is an absolute constant, then
\[
\left|M^*(A)-\frac{1}{N}\sum_{i=1}^N \|g_i\|_{\mathcal{K}(A)^*}\right| \leq \delta M^*(A)
\]
with probability $1-\beta$.
\end{proposition}


We illustrate the result of Proposition~\ref{prop:mstar-reco} by representing the lower bound on the recovery threshold for a Gaussian matrix. We see in Figure~\ref{fig:m*-gauss} that the $M^*$ based lower bound on the recovery threshold that we computed fit reasonably well the frontier between the black (never recovers) and white (always recover) area which is an upper bound on the true sparse recovery threshold

\begin{figure}
    \centering
    \includegraphics[width=0.55\textwidth]{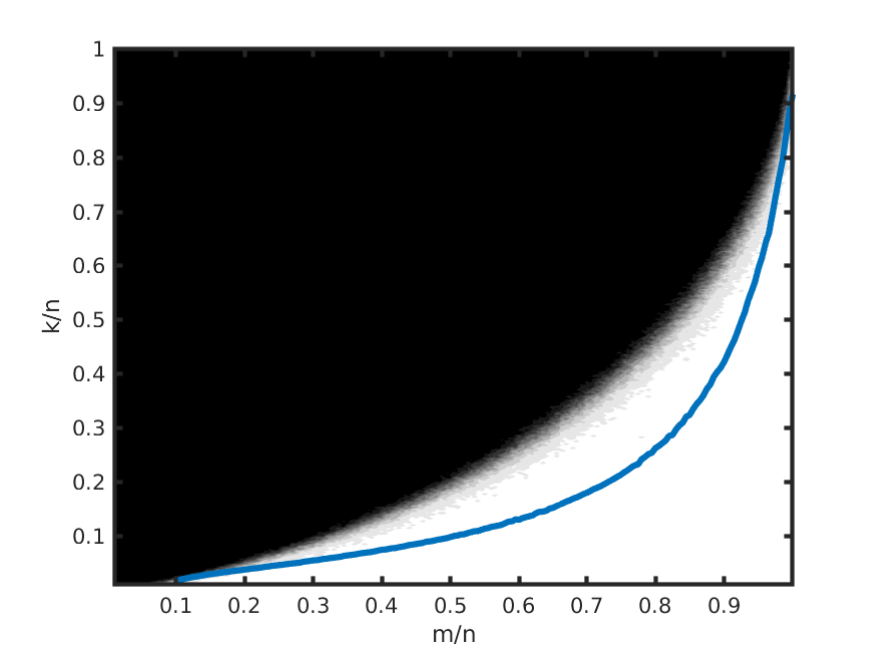}
    \caption{Let $n=200$, $G\in\reals^{n\times n}$ a Gaussian Matrix, $A_m\in\reals^{m\times n}$ the matrix constituted with the $m \leq n$ first rows of $G$. The greyscale plot in the background correspond to the empirical probability (black is $0$ and white is $1$) of recovering (with precision $10^{-8}$) a signal $x_0\in\reals^n$ with $\mathbf{Card}(x_0) = k$ by solving \eqref{eq:l1-dec} with $A=A_m$. The blue curve represents the bound $\tfrac{20}{4cM^*(A_{m-20})^2}$ of Proposition~\ref{prop:mstar-reco} with $c$ chosen such that the blue curve fit the best the black-white border. }
    \label{fig:m*-gauss}
    
\end{figure}

In the following we illustrate the interest of the $M^*$ proxy in two applications, dictionary learning and MRI sampling.

\section{Algorithms for Regularized Dictionary Learning}\label{s:algos}
We seek to use the $M^*$ quantity defined above as a penalty in various dictionary learning problems, to improve generalization performance, i.e. the performance of learned dictionaries on new data points. In general, given a dictionary learning task which involves minimizing a loss $l : A \in \reals^{m \times n} \xrightarrow{} l(A) \in \reals$ with respect to a dictionary matrix $A$, on an admissible set $\mathcal{C}$, we will focus on the penalized loss minimization problem
\BEQ\label{eq:pen-m*}
\BA{ll}
\mbox{minimize} & l(A) + \lambda M^*(A) \\
\mbox{subject to} & A \in \mathcal{C}
\EA
\EEQ
in the variable $A\in\reals^{m \times n}$.

\subsection{Alternating Minimization}\label{ss:opt-m*} 
For $F$ a matrix in $\reals^{n \times n-m}$ such that $F^TF = \idm_{n-m}$, let's define 
\BEQ\label{eq:m*-fun}
\BA{cc}
H(F) &= \Expect\left[\max_{\|Fy\|_1 \leq 1} y^Tg \right]\\
&= \Expect\left[\min_{F^Tx=g} \|x\|_\infty \right]
\EA
\EEQ
with  $g\sim{\mathcal N}(0,\idm_{n-m})$, the second line being obtained by duality.
Let $A \in \reals^{m\times n}$ be a dictionary matrix with linearly independent rows and $F$ be a orthogonal basis of $\nullspace(A)$, i.e. $F^TF=\idm_{n-m}$ and $AF=0$. $\mathcal{K}(A)$ defined as in \eqref{eq:K-ball} can be rewritten 
\BEQ\label{eq:K-ball_F}
\mathcal{K}(A) = \left\{Fy, y\in \reals^{n-m} , \|Fy\|_1 \leq 1 \right\}
\EEQ
We notice that for that choice of $F$, $H(F) = M^*(A)$, the following lemma allows to compute stochastic gradients of $H(F)$. 

\begin{lemma}\label{lem:pgrad-m*}
Given $g\in\reals^{n-m}$, the function with input $F\in\reals^{n\times n-m}$ defined as
\BEQ \label{eq:fpgrad-m*}
\BA{rll}
\nu_g(F) \triangleq & \mbox{min.} & ||x||_{\infty} \\
& \mbox{s.t.}& F^Tx  = g
\EA
\EEQ
where the minimization is performed in the variables $x \in \reals^n$, has generalized derivatives at full-rank matrices $F$ of the form 
\BEQ\label{eq:pgrad-m*}
\partial\nu_g(F) = \mathrm{Conv}\left(\{-x_g^*(F)y_g^{*T}(F),\; \text{ with } (x_g^*(F),y_g^*(F)) \text{ any primal-dual solution of } \eqref{eq:fpgrad-m*}\}\right)
\EEQ
In addition when \eqref{eq:fpgrad-m*} admits a unique primal-dual solutions $(x_g^*(F),y_g^*(F))$, $\nu_g$ is differentiable at point $F$ and $\nabla \nu_g(F) = -x_g^*(F)y_g^{*T}(F)$.
\end{lemma}
\begin{proof}
See \citep[Theorem 5.1, Proposition 5.3]{De19}.
\end{proof}
\begin{remark}
For $F$ a full-rank matrix drawn at random, primal-dual solution of Problem~\eqref{eq:fpgrad-m*} is unique almost surely. In practice we will always deal with $F$ such that the primal-dual solution is unique (by adding negligible perturbations to pathological $F$s).
\end{remark}

Introducing the variable $F$ in the formulation $\eqref{eq:pen-m*}$ leads to the new problem
\BEQ\label{eq:learn-pen}
\BA{ll}
\mbox{minimize} & l(A) + \lambda H(F)\\   
 \mbox{subject to} & A \in \mathcal{C}, AF = 0\\
&  F^TF = I
\EA\EEQ
Imposing $AF = 0$ may yield numerical issues and we can replace the hard constraint by a penalty on $||AF||_F^2$, where $||\cdot||_F$ is the Frobenius norm. The problem then becomes 
\BEQ\label{eq:learn-pen-pen}
\BA{ll}
\mbox{minimize} & l(A) + \lambda H(F) + \mu||AF||^2_F\\
\mbox{subject to} & A \in \mathcal{C}\\
 & F^TF = I 
 \EA
\EEQ
in the variables $A\in\reals^{m \times n}$ and $F\in\reals^{n \times n-m}$. 

Assuming $l(A)$ is easy to minimize, a classical technique to solve the above problem involve alternate minimization in $A$ and $F$. Thus all the algorithms described in the next parts will then follow the generic alternating minimization structure of Algorithm \ref{algo:generic} .
\vspace{-0.1cm}
\begin{algorithm}[!ht]
\caption{Generic Alternate Minimization Algorithm}
\label{algo:generic}
\begin{algorithmic}
\STATE \algorithmicrequire\; Initial $A$, number of iterations $n_{\text{iter}}$, penalties $\mu$,$\lambda$, stepsize $\tau$, number of gradient steps $n_{sgd}$.
\FOR{1 to $n_{\text{iter}}$}
\STATE Update $A := \argmin_{A \in \mathcal{C}} l(A) + \mu||AF||^2_F$,
\STATE Update $F := \textbf{SGD}_{M^*}(\nullspace(A),\tau,n_{sgd},\mu,\lambda)$, \hspace{6.5cm} (see Algorithm~\ref{algo:sgd})
\ENDFOR
\STATE \algorithmicensure\;Dictionary matrix $A$.
\end{algorithmic}
\end{algorithm}

The variable $F$ lies in the Stiefel manifold, hence minimizing in $F$ can be performed using stochastic gradient descent on the Stiefel manifold  
\BEQ\label{eq:stiefel}
\mathcal{M} =\left\{F\in\reals^{n\times n-m} : F^TF = I_{n-m} \right\}.
\EEQ
using the stochastic derivative obtained in Lemma~\ref{lem:pgrad-m*}.
Updates consist in projecting the stochastic gradient on $T_F\mathcal{M}$, the tangent space of $\mathcal{M}$ at the current point~$F$, make a gradient step in this tangent space and finally re-project the result on the manifold to get the new point. The resulting stochastic gradient descent algorithm is described in Algorithm~\ref{algo:sgd}.
\vspace{-0.1cm}
\begin{algorithm}[!ht]
\caption{Stochastic Gradient Descent on $\mathcal{M}$, $\textbf{SGD}_{M^*}(F_0,\tau,n_{sgd},\mu,\lambda)$}
\label{algo:sgd}
\begin{algorithmic}
\STATE \algorithmicrequire\; Initial $F_0 \in \reals^{n\times n-m}$, penalties $\mu,\lambda$, stepsize $\tau$, number of gradient steps $n_{sgd}$.
\STATE $F := F_0$.
\FOR{1 to $n_{sgd}$}
\STATE Sample $g \sim \mathcal{N}(0,\idm_{n-m})$,
\STATE $d := \lambda\nabla\nu_g(F) + 2\mu A^TAF $,
\STATE Update $F := \text{Proj}_{\mathcal{M}}(F -  \tau\text{Proj}_{T_F\mathcal{M}}(d))$, 
\ENDFOR
\STATE \algorithmicensure\;Matrix $F$.
\end{algorithmic}
\end{algorithm}


In practice the $SGD$ on $F$ will start at $\nullspace(A)$ instead of the current iterate of $F$.\\
Lemma~\ref{lem:pgrad-m*} could have also been applied to obtain stochastic gradient of $M^*(A)$, and one could solve \eqref{eq:pen-m*} using stochastic gradient descent directly. However in our dictionary learning experiments, results and computational time were much less satisfying with this method than when introducing $F$.  

\subsection{Compression by Dictionary Learning}\label{ss:dict-learn} Our objective here is to learn an over-complete dictionary that has good compressed sensing properties, i.e. a low $M^*$ in our setting, to optimize dictionary performance out-of-sample.

\subsubsection{Dictionary Learning}\label{sss:classic-dict} Let us recall the dictionary learning problem and its classical formulation for compression. Given a set of training image patches $Y = (Y_1,\dots,Y_m) \in \reals^{n \times m}$ and a sparsity target $S$ , the goal is to find an over-complete dictionary $D = (D_1,\dots,D_p) \in \reals^{n \times p} (n < p  \ll m)$ and a representation $X = (X_1,\dots,X_m) \in \mathbb{R}^{p \times m}$ which minimize the training loss
$$\sum_i ||Y_i - DX_i||^2_2 = ||Y-DX||^2_F.$$
with $||X_i||_0 \leq S $. One classical regularization strategy is to normalize the columns of the dictionary, often called atoms. This prevents the problem to be ill-posed with infinitely many solutions differing only by the norms of their atoms. The problem then becomes
\BEQ\label{eq:dico-classic}
\BA{ll}
\mbox{minimize} & ||Y-DX||^2_F\\
 \mbox{subject to} & ||D_i||_2 = 1,\quad\mbox{$i=1,\ldots,p$}\\
 & ||X_j||_0 \leq S,\quad\mbox{$j=1,\ldots,m$}
 \EA
\EEQ
in the variables $D\in\reals^{n\times p}$, $X\in\reals^{p\times m}$. 

A standard way to deal with problem~\eqref{eq:dico-classic} is the KSVD Algorithm \citep{Elad06}. This is an alternate minimization algorithm between the dictionary $D$ and the representation $X$. In the minimization with respect to $X$, Orthogonal Matching Pursuit (see e.g. \citep{Cai11a} is used to find an approximate solution with a given cardinality. For the minimization with respect to the dictionary $D$, the updates are made column by column. Given $X$, an update for the column $j$ then consists in solving
\BEQ\label{eq:dico-update}
\BA{ll}
\mbox{minimize} & ||Y-\sum_{i \neq j}D_iX^{(i)} - D_jX^{(j)}||^2_F\\
 \mbox{subject to} & ||D_j||_2 = 1 
 \EA
\EEQ
in the variable $D_j \in \reals^{n}$, with $X^{(i)}$ being the $i$-th row of $X$. If one allows the minimization to include the variables $(D_j ,X^{(j)})$, this comes down to finding the rank one matrix $D_jX^{(j)}$ that best approximates $M = Y-\sum_{i \neq j}D_iX^{(i)}$ in term of Frobenius norm. This can be obtain by performing a rank one SVD of $M$. One significant advantage of this method is that it directly gives a normalized update for $D_j$ and also guarantees that the dictionary updates are descent steps.
KSVD is detailed in Algorithm \ref{algo:KSVD}
\begin{algorithm}[!ht]
\caption{KSVD}
\label{algo:KSVD}
\begin{algorithmic}
\STATE \algorithmicrequire\;Initial Dictionary $D_0$, training patches $Y = [Y_1,\dots,Y_m]$, sparsity level $S$, number of iterations $n_{\text{iter}}$.
\STATE $D := D_0$.
\FOR{1 to $n_{\text{iter}}$}
\FOR{j := 1 to $m$}
\STATE $X_j := \textbf{OMP}(D,Y_j,S)$,
\ENDFOR
\FOR{l := 1 to $p$}
\STATE $\Omega := supp(X^{(l)})$, where $X^{(l)}$ is the l-th row of $X$,
\STATE $E := Y - \sum_{i \neq l}d_iX^{(i)}$,
\STATE $[U,S,V] := \textbf{SVD}(E_{\Omega})$,
\STATE $D_l := U_1$,
\STATE $X^{(l)} := S_{11}V_1^T $,
\ENDFOR
\ENDFOR
\STATE \algorithmicensure\;Dictionary $D = [d_1,\dots,d_p]$.
\end{algorithmic}
\end{algorithm}

\subsubsection{Dictionary Learning with  M* Penalization} The penalized formulation introduced in \eqref{eq:learn-pen-pen} can be used to learn a dictionary with low $M^*$. The penalized learning problem is then written
\BEQ\label{eq:dico-m*-pre}
\BA{ll}
\mbox{minimize} & ||Y-DX||^2_F + \lambda H(F) + \mu ||DF||^2_F\\
 \mbox{subject to} & ||D_i||_2 = 1,\quad\mbox{$i=1,\ldots,p$}\\
 & ||X_j||_0 \leq S,\quad\mbox{ $j=1,\ldots,m$}\\
 & F^TF = I_{p-n}
 \EA
\EEQ
in the variables $X\in\reals^{p\times m},D\in\reals^{n\times p},F\in\reals^{p\times p-n}$.

There is no change in the updates of the representation~$X$ when everything else is fixed compared to the classical setting. However for the dictionary updates in the variable~$D$, the addition of the penalty term $\mu ||DF||^2_F$ prevents the use of the SVD approach from Algorithm \ref{algo:KSVD}. Instead, the new dictionary is chosen to annihilate the gradient of the loss with respect to $D$ and then projected on the admissible set $\mathcal{C}$. If $\mathcal{C}$ is chosen to be the set of dictionaries with normalized columns as in KSVD, the projection changes the current value of $M^*$, since normalizing each column changes the nullspace of the matrix. To avoid this effect, we can take $\mathcal{C_*} = \{D\;|\; \max(||D_i||_2) = 1\}$. This set contains the previous one and the projection on it simply reduces to divide all the coefficients of the dictionary by $\max(||D_i||_2)$ which has no effect on the $M^*$. Overall, the $M^*$-regularized dictionary learning problem is then written
\BEQ\label{eq:dico-m*}
\BA{ll}
\mbox{minimize} & ||Y-DX||^2_F + \lambda H(F) + \mu ||DF||^2_F\\
 \mbox{subject to} & \max((||D_i||_2)_{i\in[1,p]}) = 1 \\
 & ||X_j||_0 \leq S,\quad\mbox{ $j=1,\ldots,m$}\\
 & F^TF = I_{p-n}
 \EA
\EEQ
in the variables $X\in\reals^{p\times m},D\in\reals^{n\times p},F\in\reals^{p\times p-n}$.

Finally, the update with respect to $F$ is done as in Algorithm~\ref{algo:generic}, using a stochastic gradient descent to minimize $M^*$ on the Stiefel Manifold. The complete $M^*$-penalized dictionary learning algorithm is then detailed as Algorithm~\ref{algo:pen-dico}.
\begin{algorithm}[!ht]
\caption{Penalized Dictionary Learning}
\label{algo:pen-dico}
\begin{algorithmic}
\STATE \algorithmicrequire\;Initial Dictionary $D_0$, Initial nullspace $F_0$, training patches $Y = [Y_1,\dots,Y_m]$, sparsity level $S$, number of iterations $n_{\text{iter}}$, regularization parameters $\mu,\lambda$, stepsize $\tau$, number of gradient iterations $n_{sgd}$
\STATE $D := D_0$, $F := F_0$.
\FOR{1 to $n_{\text{iter}}$}
\FOR{j := 1 to $m$}
\STATE $X_j := \textbf{OMP}(D,Y_j,S)$,
\ENDFOR
\STATE $D_* : = \underset{D\in\reals^{n\times p}}{\argmin\;} \|DX-Y\|^2_F + \mu\|DF\|^2_F$,
\STATE $D := \text{proj}_{\mathcal{C_*}}(D_*)$,
\STATE $F := \textbf{SGD}_{M^*}(\nullspace(D),\tau,n_{sgd},\mu,\lambda)$,
\ENDFOR
\STATE \algorithmicensure\;Dictionary $D$.
\end{algorithmic}
\end{algorithm}

\subsection{Inpainting by Dictionary Learning}\label{ss:inpainting} Inpainting is a particular class of denoising problems for imaging. This is a situation where the noise is multiplicative and takes its values in $\{0,1\}$. For an image $I$ of size \\N $\times$ N the noise matrix is called a mask denoted $B\in\{0,1\}^{N\times N}$. What is observed is a noisy version of the image $I\odot B$ (where $\odot$ is the Hadamard product of matrices) which is basically $I$ with missing parts appearing as black holes. Dictionary learning by patches has been adapted to the inpainting problem giving good results (see e.g. \citep{Mair08a}). In this section, we adapt the $M^*$ penalized algorithm to the inpainting setting.

\subsubsection{Inpainting Problems} Given some training patches $Y = [Y_1,\dots,Y_m] \in \mathbb{R}^{n \times m}$ and a mask $B = [B_1,\dots,B_m] \in \mathbb{R}^{n \times m}$, the idea of inpainting by patches is essentially the same as the classical dictionary learning principle. It seeks to find a sparse representation of the training patches using a few learned atoms. However, only $B \odot Y$ is accessible, meaning that information is only available on some pixels of each patch. Due to the intrinsic sparse structure of natural images it is reasonable to think that there is enough information in the visible pixels to learn a good dictionary to fill the masked parts of the image. The learning task in this case is then simply
\BEQ\label{eq:inpainting}
\BA{ll}
\mbox{minimize} & ||B \odot (Y-DX)||^2_F \\
 \mbox{subject to} & ||D_i||_2 = 1,\quad\mbox{$i = 1,\ldots,p$}\\
 &||X_j||_0 \leq S,\quad\mbox{$j = 1,\ldots,m$}
 \EA
\EEQ
in the variables $X\in\reals^{p\times m},D\in\reals^{n\times p}$. 

Due to the Hadamard product with $B$, the KSVD algorithm cannot be directly applied to solve this problem. \cite{Mair08a} presented a weighted KSVD algorithm that will be referred to as wKSVD in the following.
It uses an iterative method detailed in \citep{Sreb03a} to approximate a solution of the weighted rank one approximation problem encountered when trying to update the dictionary column by column as in KSVD. This consists in solving the following 
\BEQ\label{eq:rank-one-w}
\BA{ll}
\mbox{minimize} & ||W \odot (M-A)||^2_F \\
 \mbox{subject to} & \text{rank}(A) = 1
 \EA
\EEQ
with respect to the matrix $A \in \reals^{n\times m}$, with $M \in \reals^{n\times m}, W \in \reals_{+}^{n\times m}$. wKSVD is detailed as Algorithm~\ref{algo:wKSVD}.
\begin{algorithm}[!ht]
\caption{Weighted KSVD Algorithm}
\label{algo:wKSVD}
\begin{algorithmic}
\STATE\algorithmicrequire\;Initial Dictionary $D_0$, training patches $Y = [Y_1,\dots,Y_m]$, sparsity level $S$, number of iterations $n_{\text{iter}}$, number of intermediate iterations $n_{dico}$.
\STATE $D := D_0$.
\FOR{1 to $n_{\text{iter}}$}
\FOR{j := 1 to $m$}
\STATE $X_j :=\textbf{OMP}(diag(B_j)D,B_j\odot Y_j,S)$,
\ENDFOR
\FOR{l := 1 to $p$}
\STATE $\Omega := supp(X^{(l)})$,
\STATE $E :=Y - \sum_{i \neq l}d_iX^{(i)}$,
\FOR{1 to $n_{dico}$}
\STATE $E_B := B \odot E +(\ones-B)\odot d_{l}X^{(l)}$,
\STATE $[U,S,V] := \textbf{SVD}(E_{B,\Omega})$,
\STATE $d_l := U_1$,
\STATE $X^{(l)} := S_{11}V_1^T $,
\ENDFOR
\ENDFOR
\ENDFOR
\STATE\algorithmicensure\; Dictionary $D = [d_1,\dots,d_p]$.
\end{algorithmic}
\end{algorithm}
\subsubsection{M* Penalization for Inpainting} As for classical for dictionary learning, an $M^*$ penalty can be added to the classical loss, with the admissible set becoming $\mathcal{C_*} = \{D\;|\; \max(||D_i||_2) = 1\}$ and the penalized algorithm is modified using an iterative method during the dictionary update step in $D$. Indeed this update corresponds to solving the problem 
\BEQ\label{eq:dico-updt-inp-m*}
\BA{ll}
\mbox{minimize} & ||B \odot (Y-DX)||^2_F + \mu||DF||_F\\
 \mbox{subject to} & \max((||D_i||_2)_{i\in[|1;p|]}) = 1 
 \EA
\EEQ
in the variable $D \in \reals^{n \times p}$. 

One can set $Y_B = B \odot Y + (\ones - B)\odot DX$ and rewrite the loss above as $||Y_B - DX||_F + \mu||DF||_F$. The variable $Y_B$ takes the values of the training patches matrix $Y$ on the observed pixels and the current values of $DX$ on the masked ones. Minimizing $||Y_B - DX||_F + \mu||DF||_F$ with respect to $D$, with $Y_B$ fixed, can be solved as in the classical compression case. 

This procedure can be seen as a missing values estimation problem, where given a matrix of observations $Y$ with some missing values (the values of the masked pixels), one tries to find a dictionary $D$ that minimize the previous error. Hence setting $Y_B$ as detailed above consists in an estimation step where the missing values are replaced by the current estimate $DX$. Then one performs a minimization step on $D$ to update the current estimate. This is done in an iterative setting and pseudo code for $M^*$ penalized inpainting is described in Algorithm \ref{algo:pen-inpainting}.

\begin{algorithm}[!ht]
    \caption{Penalized Dictionary Learning for Inpainting}
    \label{algo:pen-inpainting}
    \begin{algorithmic}
    \STATE \algorithmicrequire\;Initial Dictionary $D_0$, initial nullspace $F_0$, training patches $Y = [Y_1,\dots,Y_m]$, mask for on the training patches $B = [B_1,\dots,B_m]$, sparsity level $S$, number of iterations $n_{\text{iter}}$, regularization parameters $\mu,\lambda$ , stepsize $\tau$, number of gradient iterations $n_{sgd}$, number of intermediate iterations $n_{dico}$.
    \STATE $D := D_0$, $F := F_0$.
    \FOR{1 to $n_{\text{iter}}$}
    \FOR{j := 1 to $m$}
    \STATE $X_j := \textbf{OMP}(diag(B_j)D,B_j\odot Y_j,S)$,
    \ENDFOR
    \FOR{1 to $n_{dico}$}
    \STATE $Y_B := B\odot Y + (\ones - B)\odot DX$,
    \STATE $D_* := \underset{D\in\reals^{n\times p}}{\argmin\;} \|DX-Y_B\|^2_F + \mu\|DF\|_F^2$,
    \STATE $D := \text{proj}_{\mathcal{C_*}}(D_*)$,
    \ENDFOR
    \STATE $F := \textbf{SGD}_{M^*}(null(D),\tau,n_{sgd},\mu,\lambda)$,
    \ENDFOR
    \STATE \algorithmicensure\; Dictionary $D$.
    \end{algorithmic}
    \end{algorithm}

\section{Numerical Experiments on Dictionary learning}\label{s:numres}
This section is dedicated to experimental results obtained using the previously described framework. The optimization toolbox Manopt \citep{Boum14a} was used to perform stochastic gradient descent on the Stiefel manifold, together with the SPAMS toolbox \citep{Mair09} to perform the OMP algorithm. All the tests were done on grayscale images of size $512 \times 512$. The size of the patches has been set to $8 \times 8$, meaning $n = 64$. The initial dictionaries $D_0$ are normalized independent mean zero and unit variance Gaussian vectors except for the the last one being the constant vector of unit norm which remains unmodified by the various algorithms to capture the mean information. The ratio $\frac{\mu}{\lambda}$ is set to $10^{-4}$, $n_{sgd} = 100$ steps and the stepsize $\tau$ of the SGD with Manopt is fixed to $10^{-2}$.

\subsection{Compression Experiments}
Compression experiments are detailed in Section~\ref{ss:comp-exp-sup} in the supplementary materials. The compression problem being much simpler than the inpainting one, we couldn't observe significant improvement when using our $M^*$ regularization strategy on it. Nonetheless, we observe a better robustness to bad initialization of the dictionary when adding $M^*$ penalty.

\subsection{Inpainting Experiments}
Following the setting of Section \ref{ss:inpainting}, with the number of atoms in the dictionaries set to $p = 128$ and the number of training patches to $m = 150p$. Experiments are based on a set of 14 gray scale $512 \times 512$ images and two masks, one representing cracks and one with text (see Figure~\ref{fig:inpaint-images}). For simplicity, in all these inpainting experiments, the regularization parameter $\mu$ has been fixed to $10^8$,  the same order of magnitude as the training loss $||Y-DX||^2_F$ for our experiments. Of course, results would further improve with $\mu$ chosen adaptively.

Reconstructing a masked image given a dictionary $D$ is done by solving \eqref{eq:inpainting} with respect to $X$ only, with $B \odot Y$ the matrix of all the overlaying patches of the masked image. 

\begin{figure}[!ht]
    \centering
    \begin{tabular}{ccc}
        \includegraphics[scale=0.865]{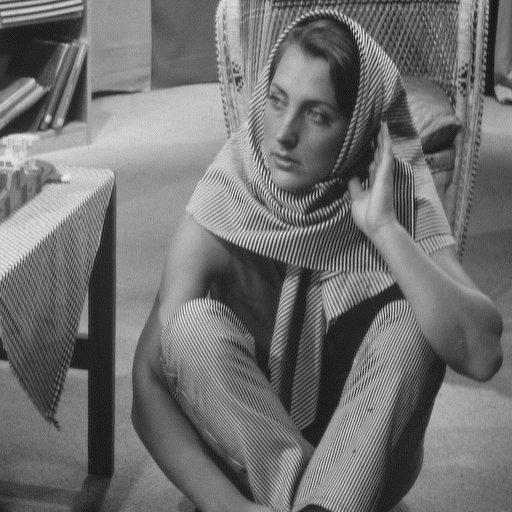}&
        \includegraphics[scale=0.122]{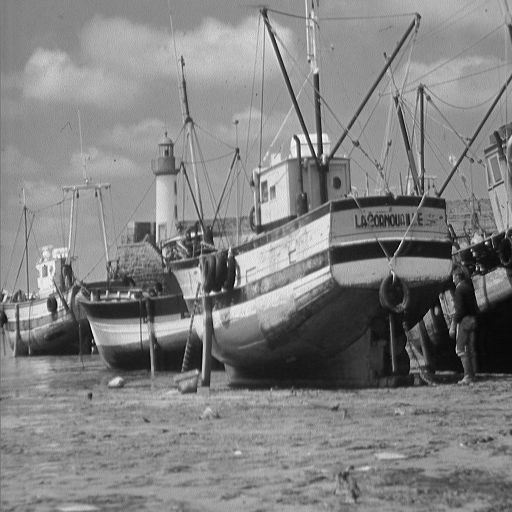}&
        \includegraphics[scale=0.122]{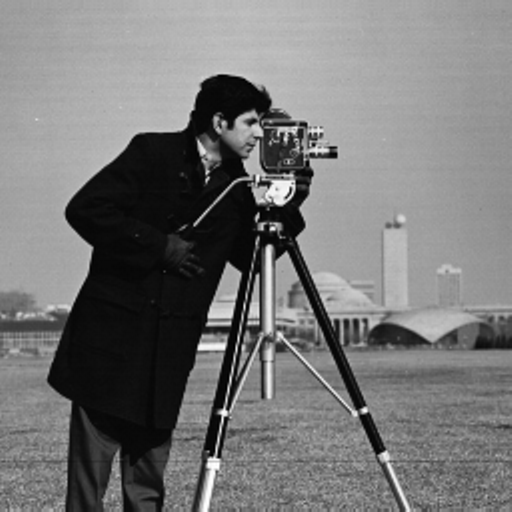}\\
        \includegraphics[scale=0.122]{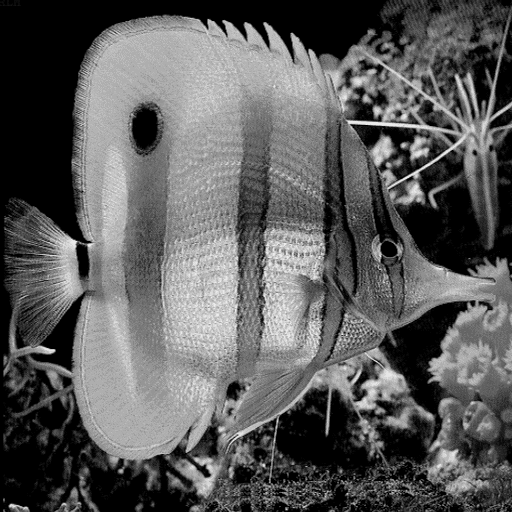}&
        \includegraphics[scale=0.122]{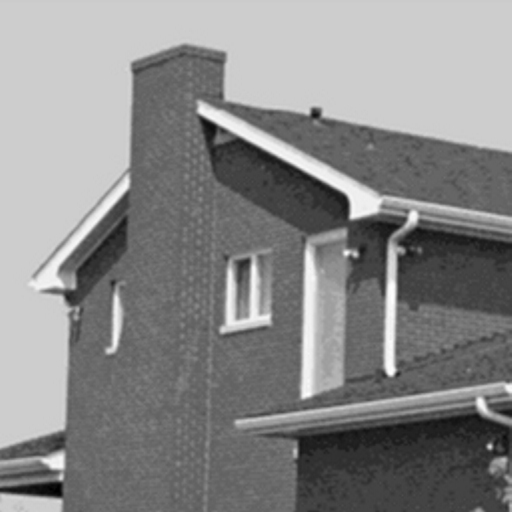}&
        \includegraphics[scale=0.122]{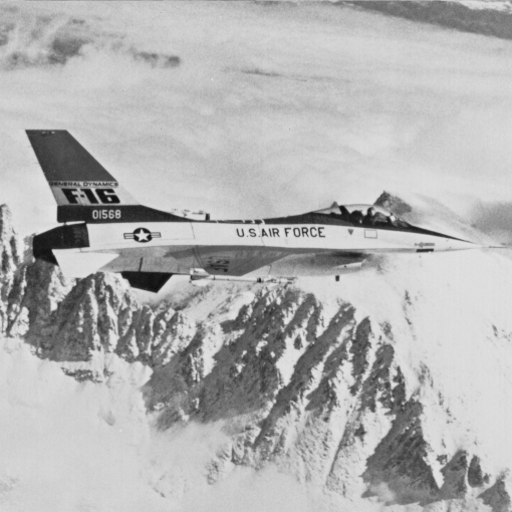}\\
        \includegraphics[scale=0.122]{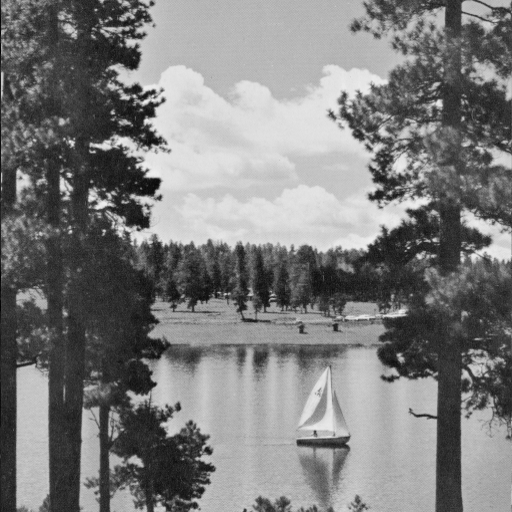}&
        \includegraphics[scale=0.122]{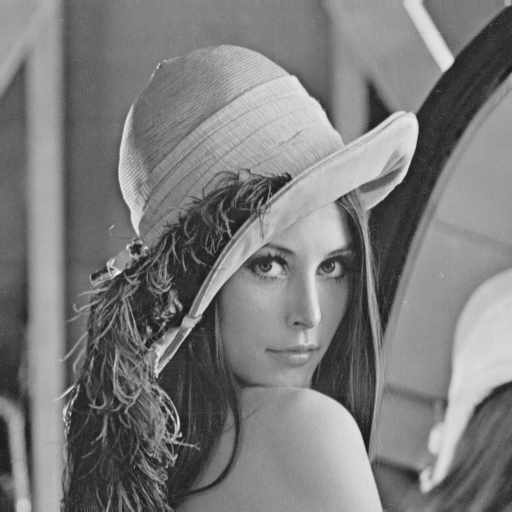}&
        \includegraphics[scale=0.122]{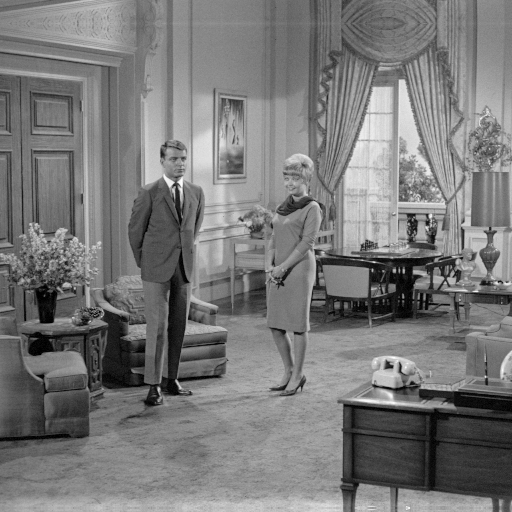}\\
        \includegraphics[scale=0.122]{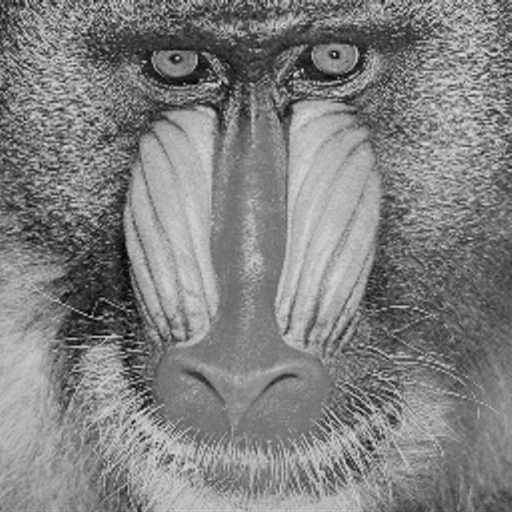}&
        \includegraphics[scale=0.122]{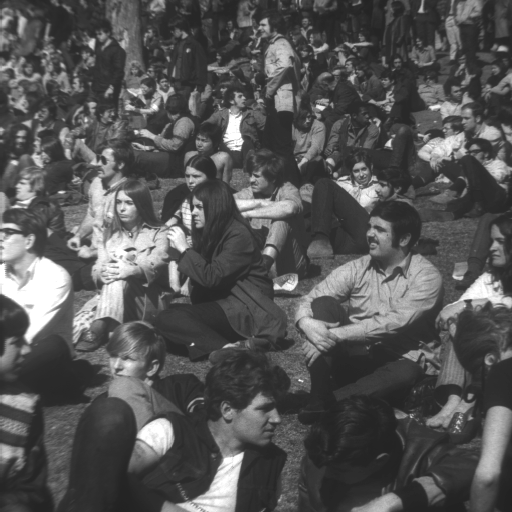}&
        \includegraphics[scale=0.122]{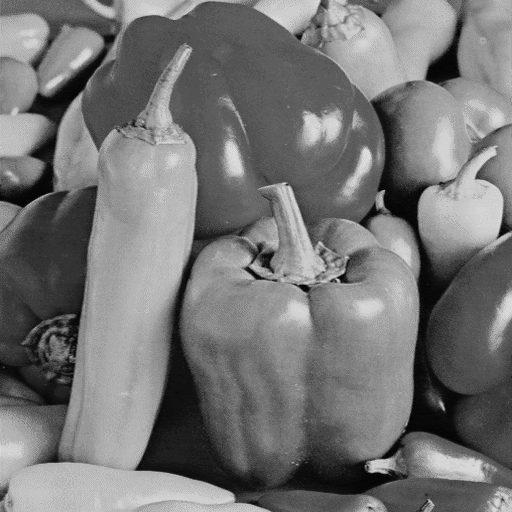}\\
        \includegraphics[scale=0.122]{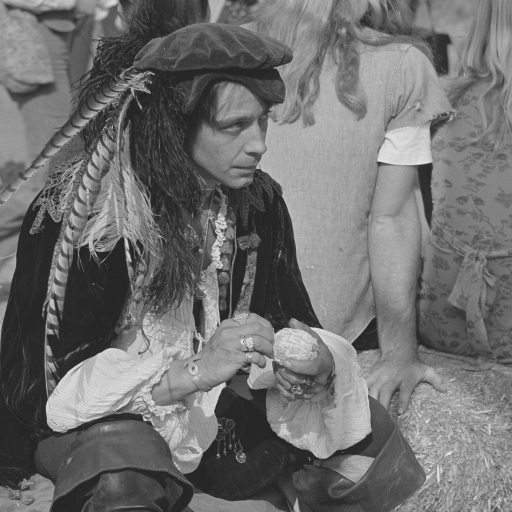}&
        \includegraphics[scale=0.122]{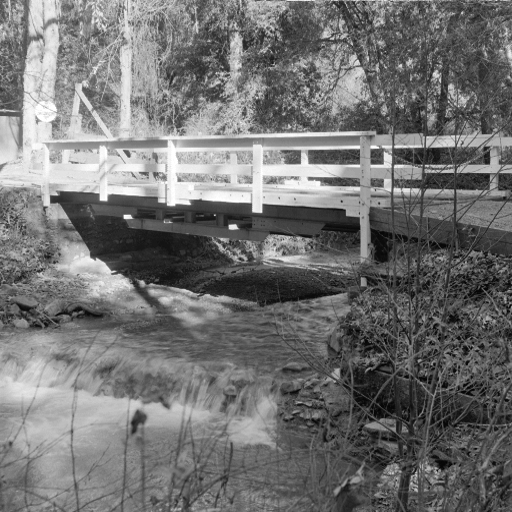}&
        \includegraphics[scale=0.122]{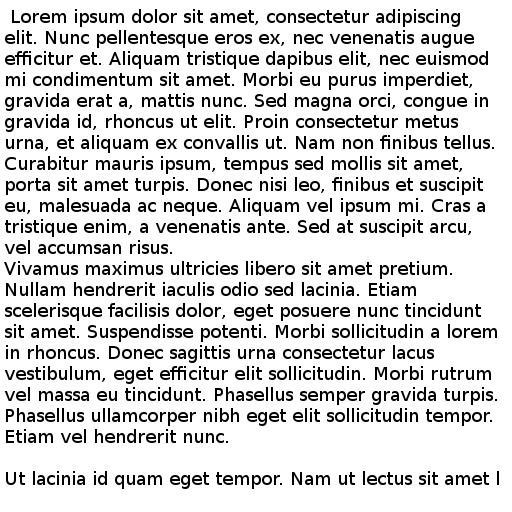}\\
        &\includegraphics[scale=0.122]{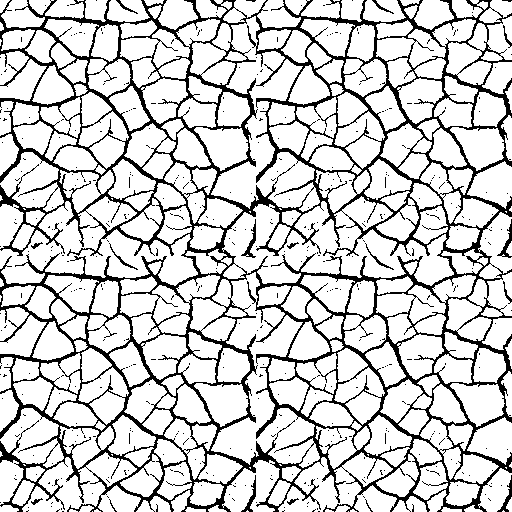}&

    \end{tabular}
    \caption{Images and Masks Used for Inpainting Experiments. From Top Left to Bottom Right Line by Line : "barbara", "boat", "cameraman", "fish", "house", "jetplane", "lake", "lena", "living-room", "mandril", "people", "peppers", "pirate", "walkbridge", "text", "cracks".}
    \label{fig:inpaint-images}
    \vspace{-0.2cm}
\end{figure}

For a given mask $B$, a given image $I$ and a training sparsity $S$, $m$ masked patches are selected randomly in the image $B\odot I$. Then both wKSVD and $M^*$ penalized algorithms are run for $50$ iterations on these training patches with a training sparsity $S$ to obtain the dictionaries $D_S(I,B)$ and $D_S^{\mu}(I,B)$.

To reconstruct a new masked image $B'\odot I' \in \reals^{512\times 512}$ thanks to a dictionary $D$, all the $(512-8)\cdot(512-8) = 254016$ patches of $I'$  are gathered in a matrix $Y = \reals^{64\times 254016}$. Compared with the compression setting, all the $8\times 8$ patches are used for the reconstruction. For simplicity, $B' \odot Y$ will represent the matrix of masked patches, even if $B'$ has not the right dimension, and corresponds to the set of all patches of the mask. We set
\BEQ\label{eq:rep-fun-inp}
\BA{rll}
Y_k(D) &\triangleq DX   &\\
 X &= \mbox{argmin.} & ||B'\odot(Y-DX)||^2_F\\
 &\mbox{s.t.} & ||X_j||_0 \leq k,\quad\mbox{$j=1,\ldots,25016$}
 \EA
\EEQ
$Y_k(D)$ is the approximation of the patches $ Y$ through $D$ with a reconstruction sparsity of $k$ when only $B' \odot Y$ is observed. An approximation $I'_{B'}$ of $I'$ is then reconstructed from patches $Y_k(D)$ by recasting them to an image and simply averaging their overlaying parts. The final reconstructed image is $B'\odot I' + (\ones - B')\odot I'_{B'}$ since $I'$ was already known on the pixels $p$ where $B'(p) = 1$ and the approximation $I'_{B'}$ is only used for the pixels where $B'(p) = 0$. Figure~\ref{fig:inp-text} shows example of masked reconstructed image using dictionaries learned by wKSVD and $M^*$ penalization.

\begin{figure}[!ht]
\centering
    \begin{tabular}{cc}
        \includegraphics[width = 0.4\linewidth,height=0.33\linewidth]{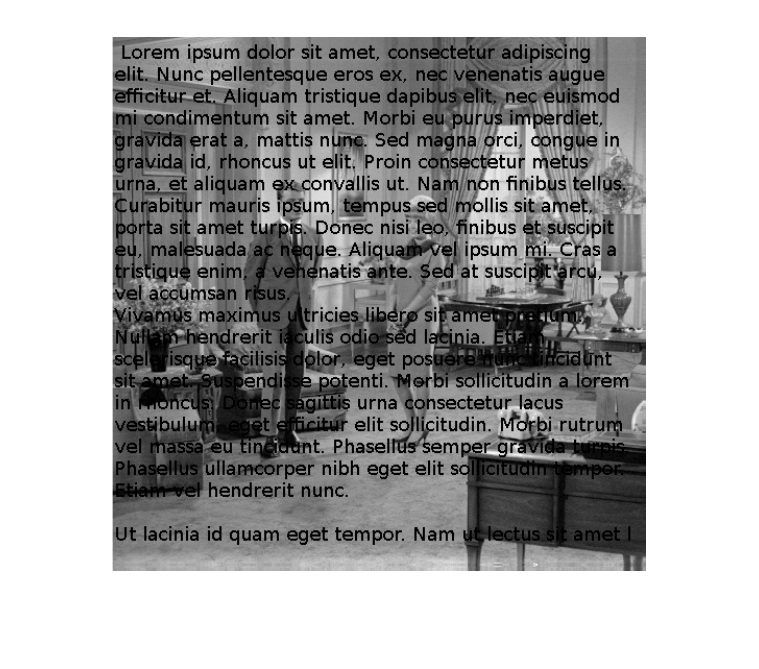}&\includegraphics[width = 0.4\linewidth,height=0.33\linewidth]{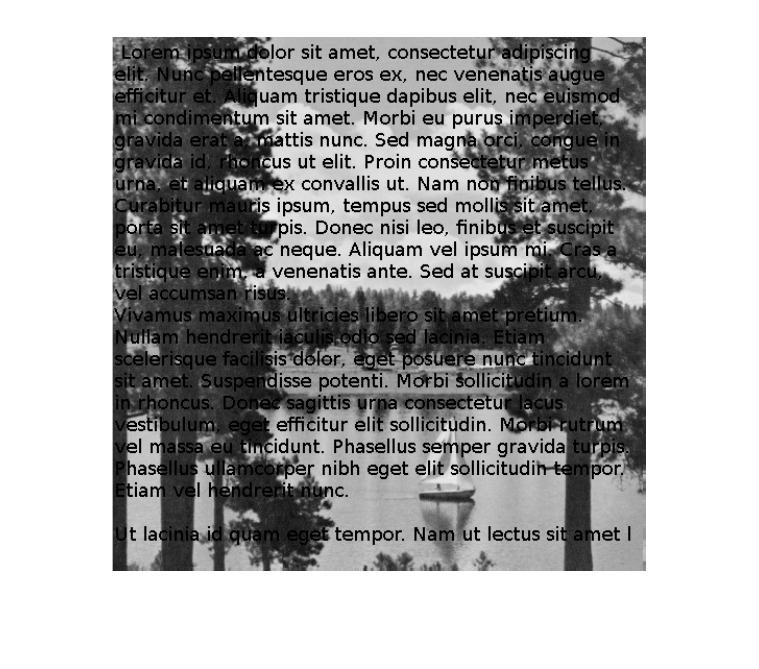}\vspace{-1cm}\\
        \vspace{-0.6cm}
        \includegraphics[width =0.4\linewidth,height=0.33\linewidth]{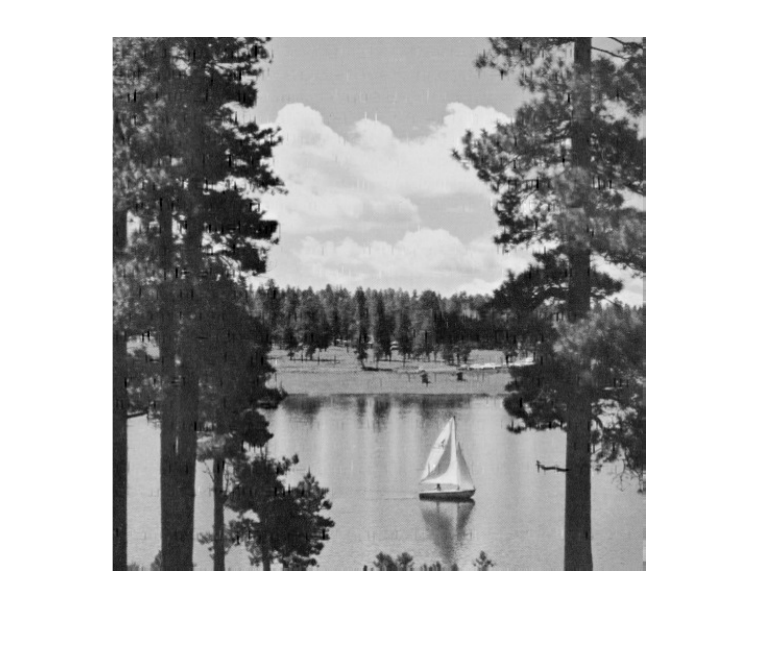}&\includegraphics[width = 0.4\linewidth,height=0.33\linewidth]{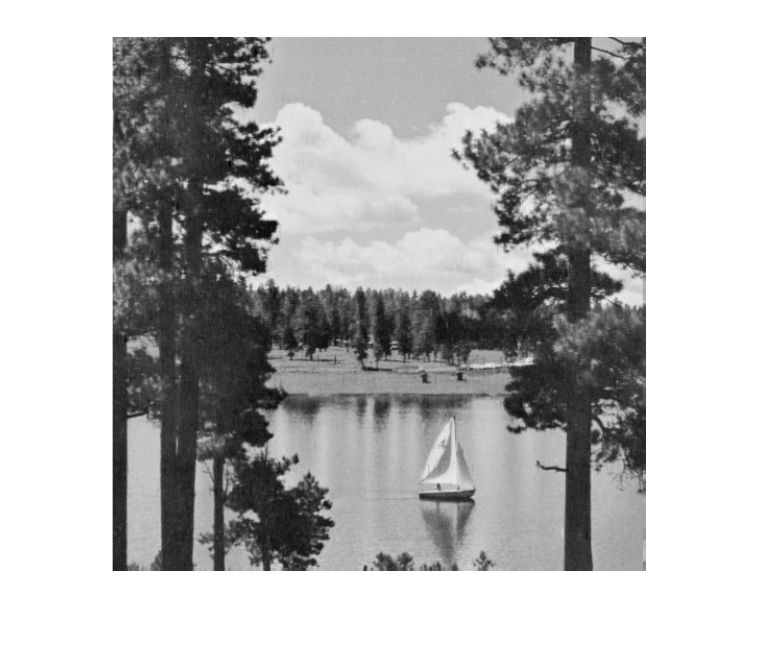}\\
        \hspace{0.02cm}\includegraphics[width = 0.28\linewidth,height=0.25\linewidth]{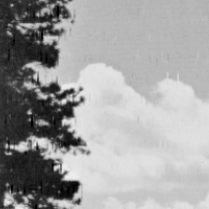}&\hspace{-0.02cm}\includegraphics[width = 0.28\linewidth,height=0.25\linewidth]{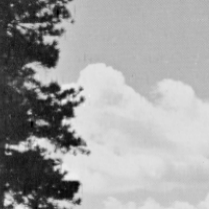}
    \end{tabular}
    \caption{Inpainting of 'lake' masked by the text mask, with dictionaries trained on 'living-room', $S = 5$, reconstruction sparsity is 16.
    Top: Masked 'living-room' and 'lake'. Middle Left: Reconstructed image using wKSVD dictionary , PSNR $= 30.8$dB. Middle Right: Reconstructed image using $M^*$ penalized dictionary, PSNR = $35.4$dB ($\mu = 10^9$). Bottom: Zoom on the upper left side of the image. One can see much more aberrations in the reconstructed image with wKSVD dictionary (Left) than with $M^*$ penalized (Right).}
    \label{fig:inp-text}
\end{figure}

Results for inpainting have been presented for $S$ between $4$ and $10$ with a regularization parameter $\mu = 10^8$. When using the same value of $\mu$ for smaller $S$ as $2$ and $3$ performances can be worse than with wKSVD, however decreasing the regularization parameter to $\mu = 10^6$ allows to retrieve or improve the reconstruction performance of wKSVD.

The new reconstructed image can be compared to the original using PSNR and SSIM measures. For instance Figure~\ref{fig:inpainting_error_curve} represents the curves of PSNR and SSIM versus reconstruction sparsity for the inpainting of the "lake" image trained on the image "living-room" and masked by the cracks mask with $S = 5$.

\begin{figure}[!ht]
    \centering
    \begin{tabular}{cc}
       \includegraphics[scale=0.35]{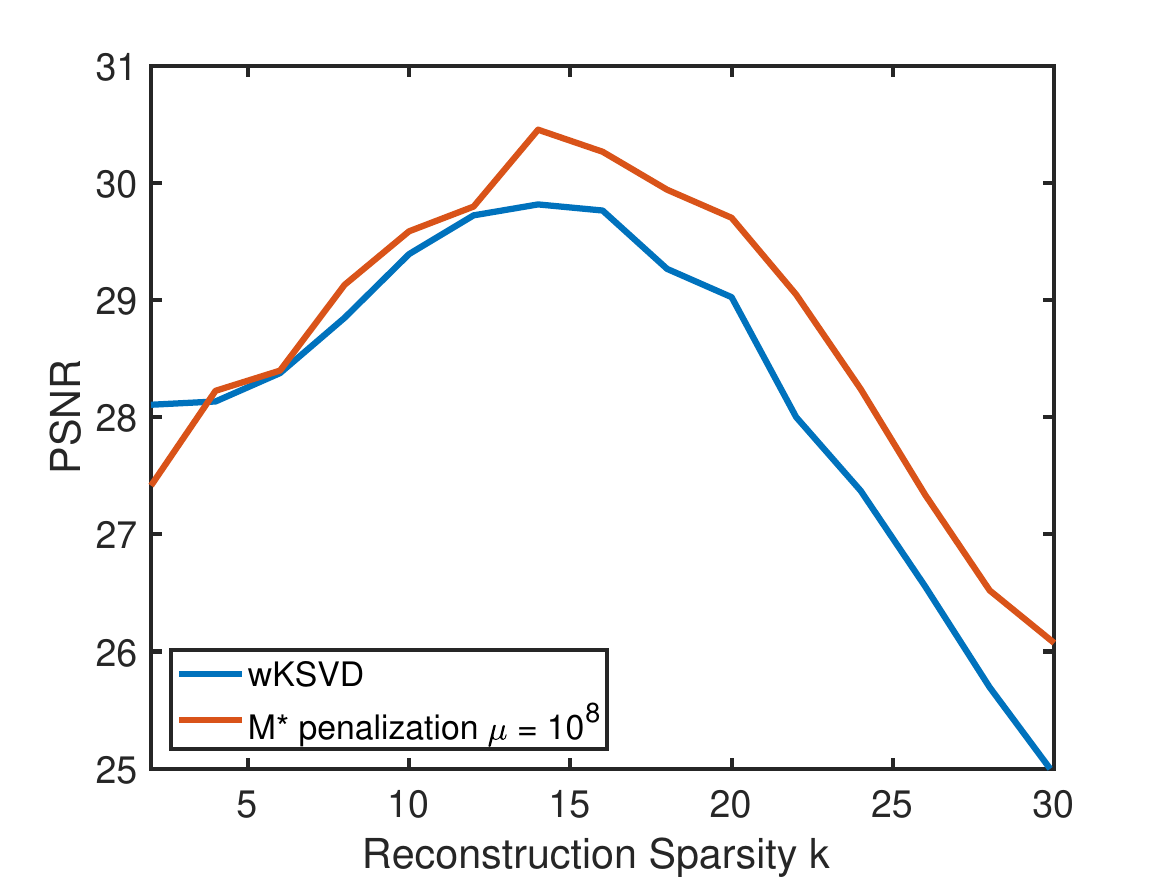} &\includegraphics[scale=0.35]{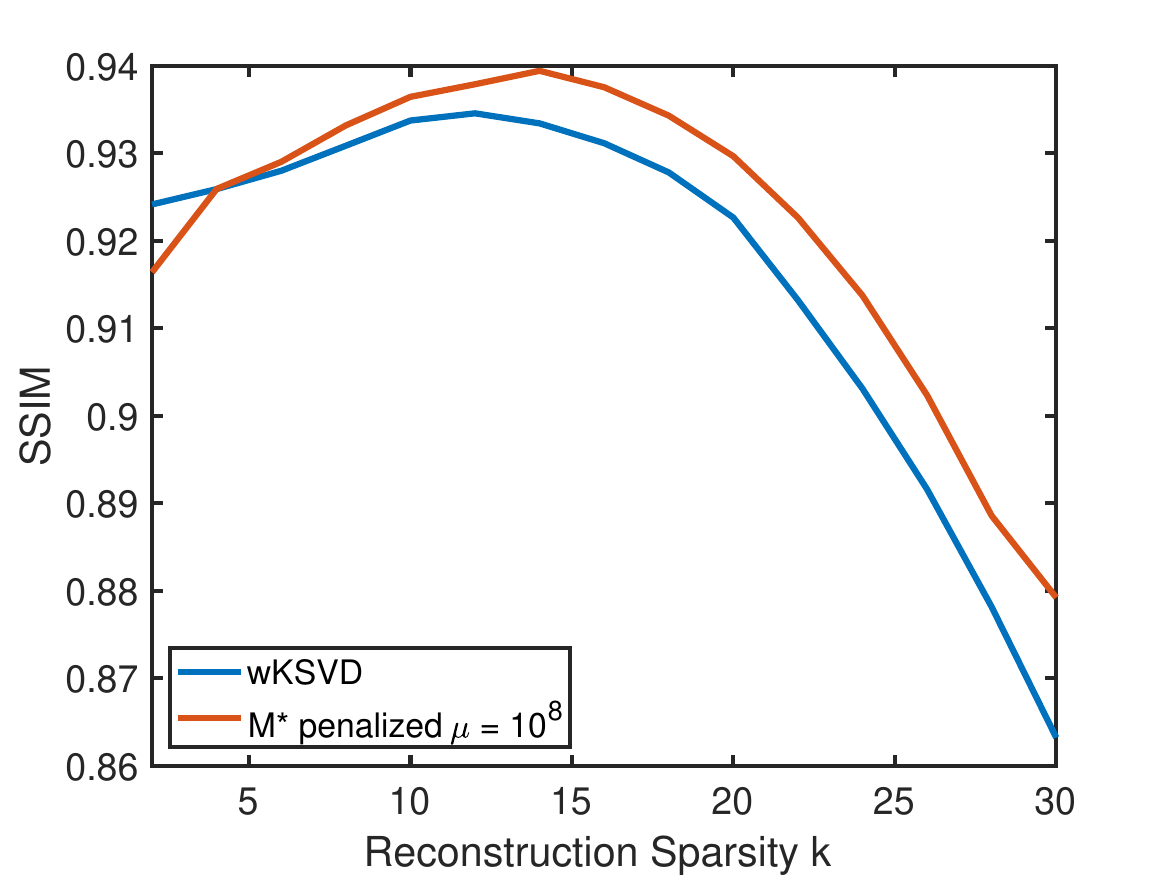} 
    \end{tabular}
    \caption{Plots show the reconstruction performances for the inpainting of 'lake' based on dictionaries learned on 'living-room' using wKSVD and $M^*$ regularized dictionary learning (taking $S = 5$). }
    \label{fig:inpainting_error_curve}
\end{figure}

The 14 images are used successively as training image with both masks and for all the sparsity levels $S$ between 4 and 10. This means $2\times14\times7$ dictionaries computed for each method. The dictionaries obtained by both methods are then used to reconstruct each of the 14 images with reconstruction sparsity $k$ between 2 and 30. This means forming (number of mask: $2 )\times $(number of train images: $14)\times$(number of train sparsity S: $7)\times$(number of test images: $14) = 2744$ PSNR and SSIM curves (as in Figure~\ref{fig:inpainting_error_curve}). Table~\ref{tab:times} in the next subsection indicates that the $M^*$ penalization doesn't involve more computational efforts, indeed the computations involved by the truncated SVDs in wKSVD are heavier than those of SGD on the nullspace in our penalized algorithm leading to a time reduction from $400$ to $200$ seconds for computing the dictionaries used to obtain Figure~\ref{fig:inpainting_error_curve}. 

For each plot, the area of the gap between the curve obtained from $M^*$ penalization and the curve from wKSVD gives an indicator of the benefit of regularized methods (the larger the better). This area can be computed as the mean of the difference between the curves. These areas are aggregated over the training sparsity $S$ and over the two masks for each couple of train/test images and the results in terms of PSNR are shown in Figure~\ref{fig:inpainting_heatmap} where the abscissa corresponds to the training images and the ordinate to the test ones.

\begin{figure}[!ht]
    \centering
    \includegraphics[width=0.55\linewidth]{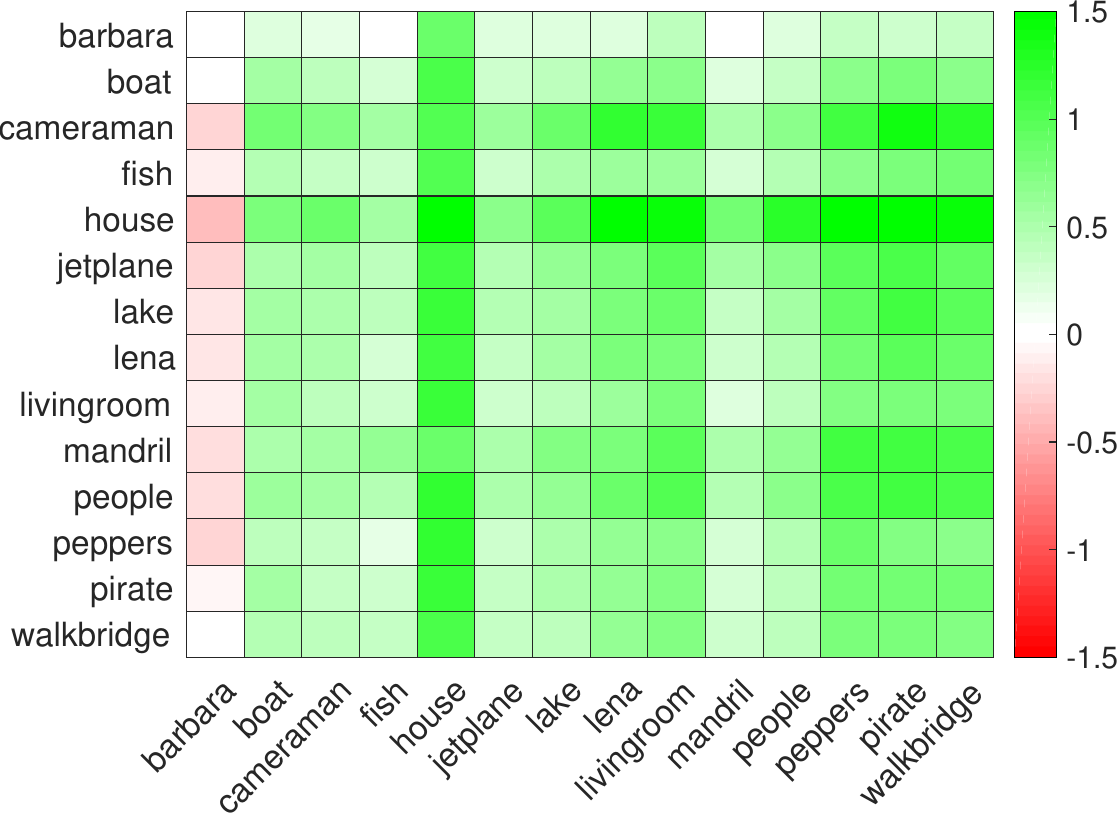}
    \caption{Heatmap of the average PSNR gap between the $M^*$ penalized method and wKSVD with training images on the x-axis and test images on the y-axis. Except for the Barbara image which has very particular texture, $M^*$ penalized algorithm performs consistently better.}
    \label{fig:inpainting_heatmap}
    \vspace{-0.2cm}
\end{figure}

Finally the distributions of the PSNR and SSIM gaps between the reconstructed images by the two methods are represented on Figure~\ref{fig:inp-gaps}. These distributions are shifted on the positive side showing globally better reconstruction performances with $M^*$ penalized method. One can also observe the distribution of the $M^*$ of the dictionaries produced by both methods in Figure~\ref{fig:inp-gaps} on the right. The x-axis corresponds to the difference of $M^*$ between the dictionaries and Gaussian matrices. Most of the dictionaries obtained by the penalized method have a nearly optimal $M^*$.

\begin{figure}[!ht]
    \centering
    \begin{tabular}{ccc}
    \hspace{-0.25cm}\includegraphics[width=0.32\linewidth,height=0.29\linewidth]{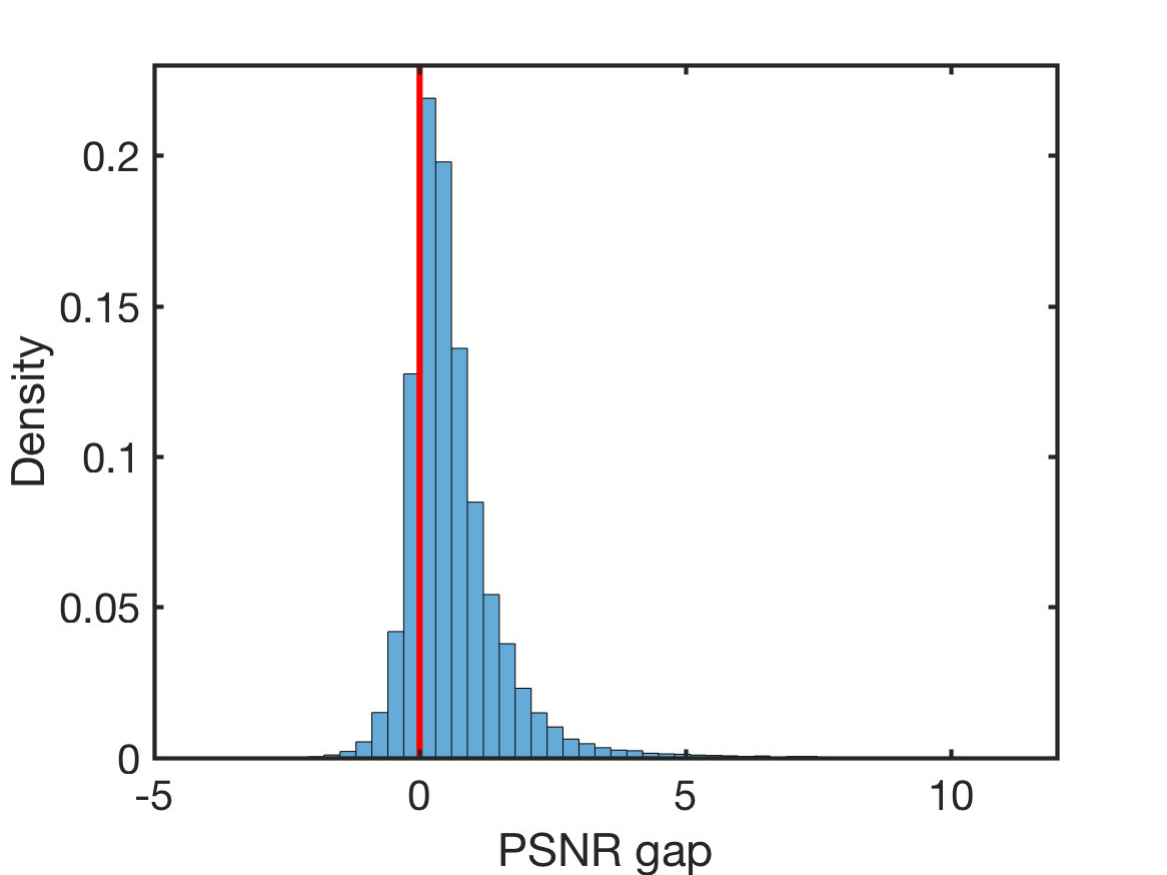} &
        \includegraphics[width=0.32\linewidth,height=0.29\linewidth]{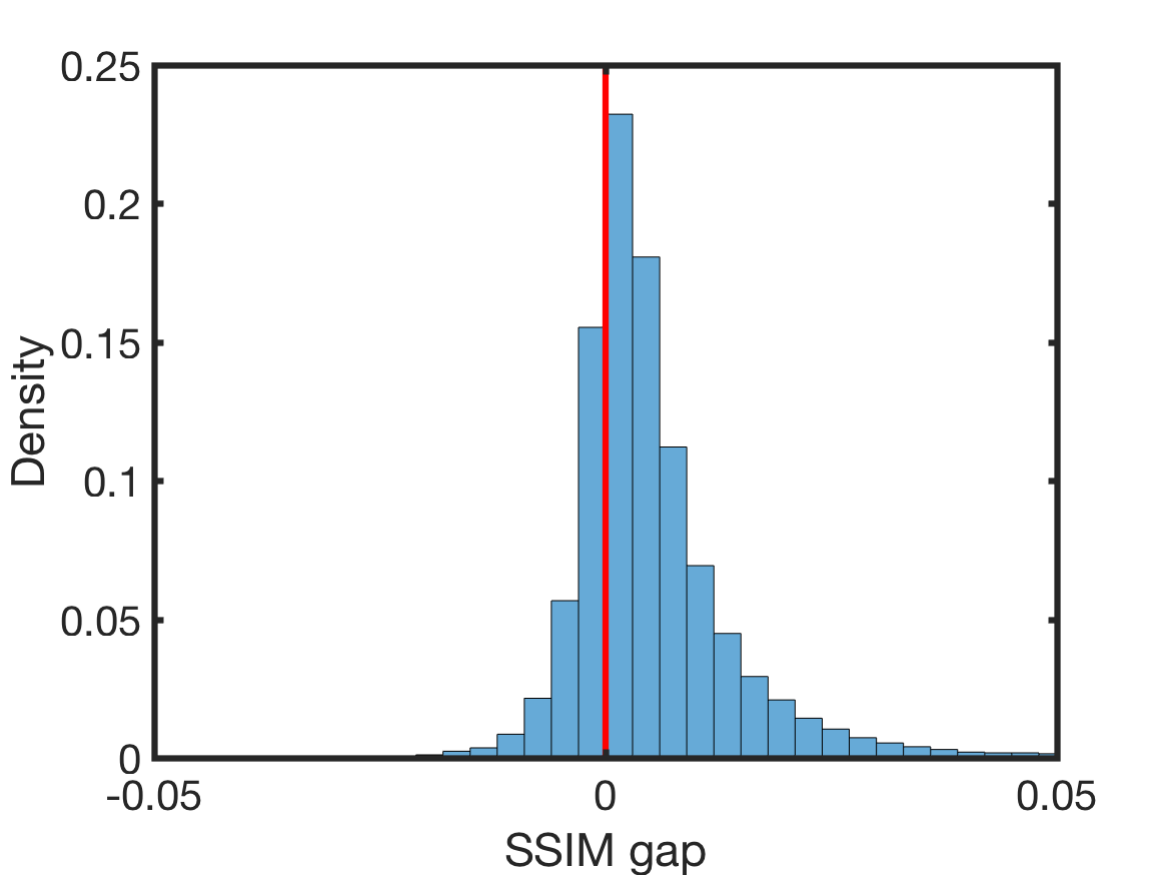} & \includegraphics[width=0.32\linewidth,height=0.29\linewidth]{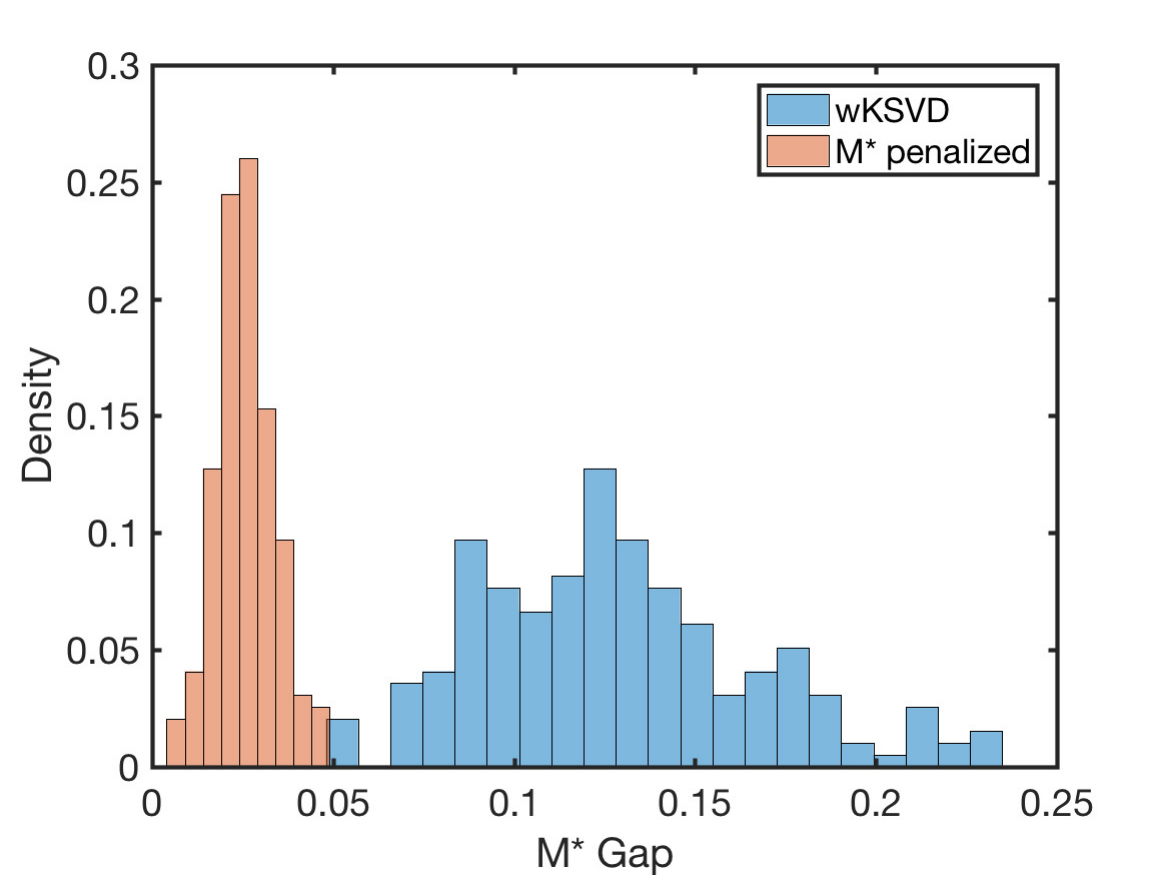}
    \end{tabular}
    \caption{Results of the $2\times7\times14\times14$ inpainting experiments. Left: Histogram of the gaps ($M^*$ regularized minus wKSVD) in SSIM of reconstructed images between both methods. Right: Histogram of the gap between $M^*$ of learned dictionaries and Gaussian $M^*$}
    \label{fig:inp-gaps}
\end{figure}
Here the training masked patches should not be fitted exactly. Indeed it would mean that the dictionary learned the noise and the reconstructed image would display artifact from the mask. In this setting the dictionary learned must have good generalization performance in order to fill the void in the images. The regularization parameter $\mu$ has been set to a constant value for all the inpainting experiments. It could of course be fine tuned to improve reconstruction results. 

\subsection{Comparaison to INK-SVD}

We also compare our results to the INK-SVD strategy \citep{Mail12}, which consists in adding a decorrelation step on the column of the dictionary after each KSVD (resp. wKSVD) update.

We recall that the coherence of a dictionary $D$ is defined as
\begin{equation}\label{eq:coherence}
\mu(A) := \underset{i \not = j}{\max}|\langle \frac{A_i}{\|A_i\|_2},\frac{A_j}{\|A_j\|_2} \rangle|
\end{equation}
The coherence gives a lower bound on the recovery threshold of a given matrix $A$. For instance \citep{Trop07} gives the following bound 
\begin{equation}\label{eq:bound-coh}
k(A) \geq \frac{1}{2}(1+\frac{1}{\mu(A)})
\end{equation}
However as mention in Section \ref{s:intro}, given a sensing matrix $A$ with sparse recovery threshold $k(A)$, coherence only allows to certify recovery of signals of cardinals lower than $O(\sqrt{k(A)})$.

INK-SVD is an algorithm from \citep{Mail12} which involves decorrelation of the columns of a dictionary in order to set an upper bound on its coherence. It can be seen as a coherence based regularization of dictionary learning. We used it in inpainting experiments by adding decorrelation steps after each dictionary update in wKSVD. Comparison between wKSVD, INK-SVD and $M^*$ penalization are presented in Figure~\ref{fig:inp-gaps-corr}. 
\begin{table}[!ht]
\vspace{-0.3cm}
\caption{Computational times taken by the three algorithms to compute the dictionaries learned on 'living-room' with training sparsity $S=5$.}
\label{tab:times}
\vskip 0.15in
\begin{center}
\begin{small}
\begin{sc}
\begin{tabular}{lccc}

& \lowercase{w}ksvd & ink-svd & $M^*$-reg \\

time(s) & 444 & 495 & 196  \\
\hline
\end{tabular}
\end{sc}
\end{small}
\end{center}
\vskip -0.1in
\end{table}
Finally Table~\ref{tab:times} displays the times of computing dictionaries with the different algorithms. We observe that $M^*$ penalization is more than two times faster, as it does not involves truncated SVDs. INK-SVD is also computationally heavier since it adds decorrelation steps to classical wKSVD.

In our experiment we used INK-SVD with the coherence of the dictionaries bounded by $0.8$. This value appeared to give the best results in the general case. When comparing $M^*$ regularization with penalization $\mu = 10^8$ against INK-SVD, results were balanced with slightly better performances from our $M^*$ penalized algorithm. When using a stronger regularization $\mu = 10^9$, Figures~\ref{fig:inp-gaps-corr} shows that $M^*$ regularization globally outperforms INK-SVD. 
In addition, it can be observed in Figure~\ref{fig:nus-vs-corr} that dictionaries obtained by INK-SVD have very bad reconstruction performances when only using a small amount of patches to reconstruct the images, whereas the $M^*$ is often consistently good with small and larger reconstruction sparsity. We can also note that inpainting results appear to be robust in the choice of the regularization parameter $\mu$ in the $M^*$ penalized algorithm whereas performances change drastically with the coherence bound in INK-SVD.

\begin{figure}[!ht]
\begin{minipage}[c]{0.45\linewidth}
\centering
    \begin{tabular}{c}
    \includegraphics[width=0.88\linewidth,height=0.75\linewidth]{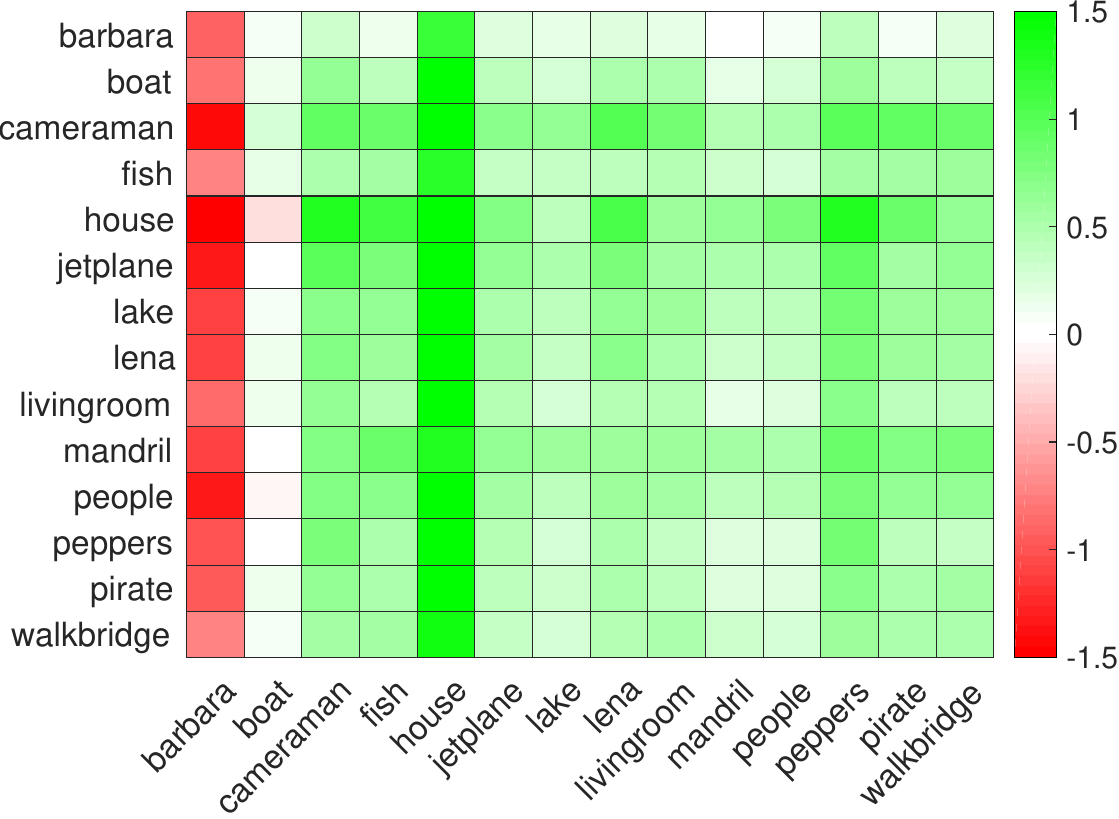}\\
    \hspace{0.5cm}\includegraphics[width=0.88\linewidth,height=0.75\linewidth]{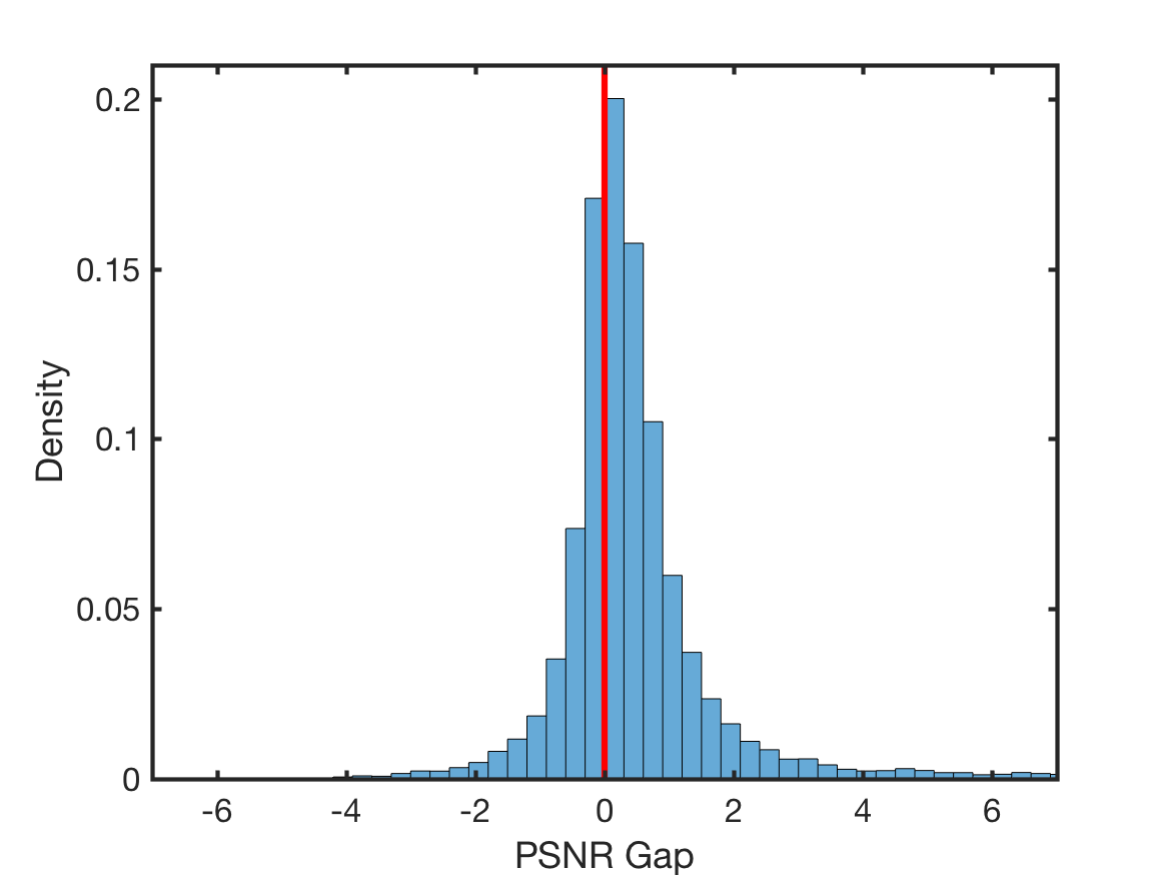} \\
    \hspace{0.5cm}\includegraphics[width=0.88\linewidth,height=0.75\linewidth]{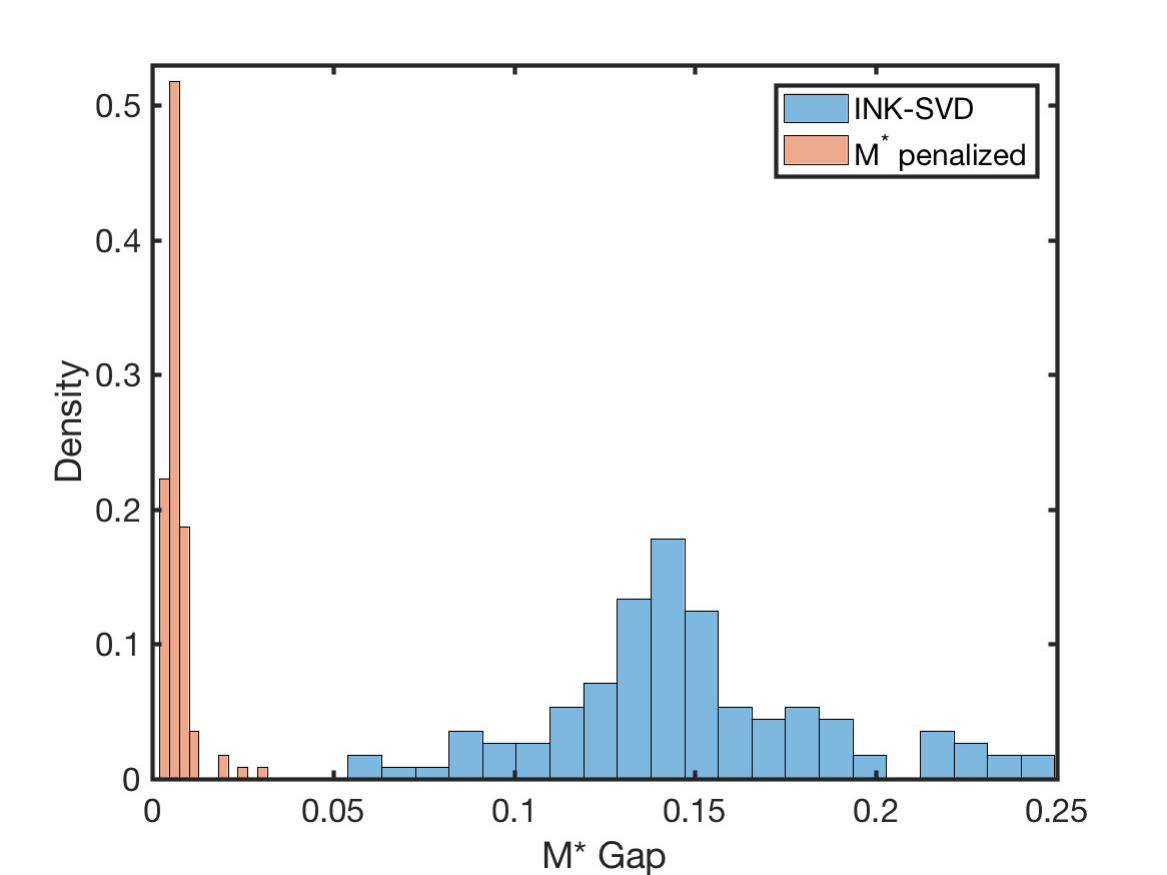}
    \end{tabular}
    \caption{Top: Heatmap of the average PSNR gap between the $M^*$ penalized method and INK-SVD with training images on the x-axis and test images on the y-axis. Middle: Histogram of the gaps ($M^*$ minus INK-SVD) in PSNR of reconstructed images between both methods. Bottom: Histogram of the gap between $M^*$ of learned dictionaries and Gaussian $M^*$.
    ($\mu = 10^9$ for $M^*$ penalization and coherence bounded by $0.8$ in INK-SVD)}
    \label{fig:inp-gaps-corr}
\end{minipage}
\hfill
\begin{minipage}[c]{0.45\linewidth}
    \centering

    \begin{tabular}{c}
        \includegraphics[width=0.88\linewidth,height=0.8\linewidth]{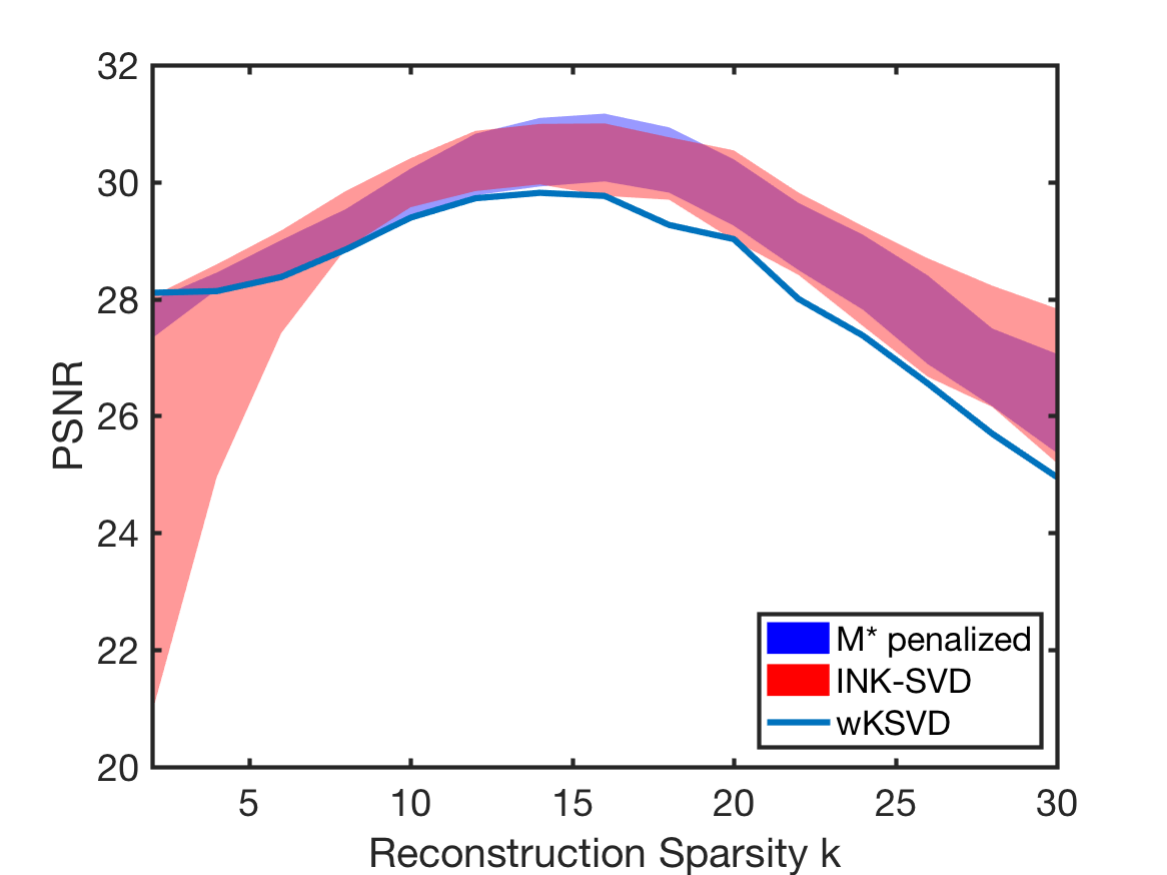}\\
        \includegraphics[width=0.88\linewidth,height=0.8\linewidth]{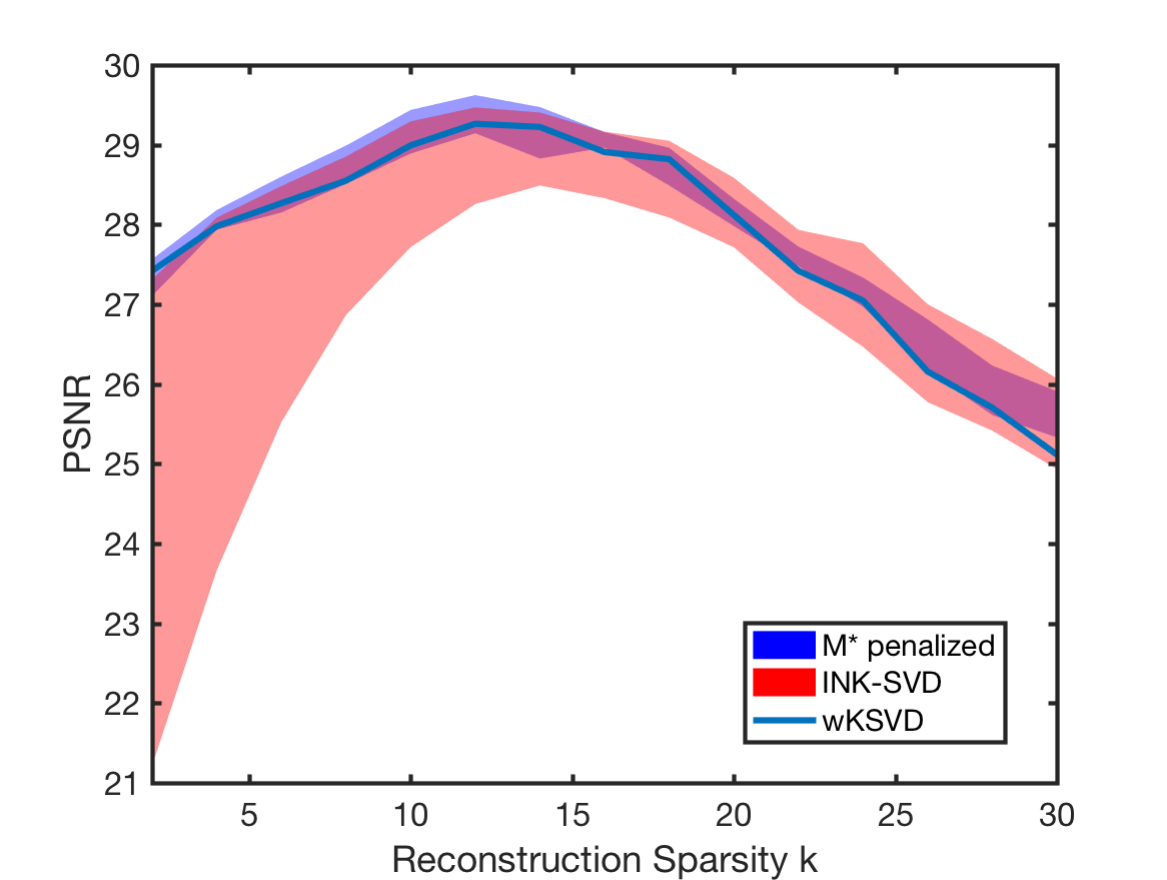}
    \end{tabular}
    \vspace{5.4cm}
    \caption{Reconstruction PSNR of the different methods versus the reconstruction sparsity. Shaded regions are delimited by the best and worst reconstruction PSNR for each algorithm when varying their regularization parameter ($\mu$ from $10^5$ to $10^{11}$ f.3or $M^*$ penalization and coherence between $0.6$ and $1$ for INK-SVD). Left: Trained on 'living-room' test on 'lake', $S = 5$ . Right: Trained on 'jetplane' test on 'fish', $S=6$.)   }
    \label{fig:nus-vs-corr}
\end{minipage}%
\end{figure}

In the next section we use the $M^*$ proxy to design sampling sets for MRI reconstruction.

\section{Using the M* proxy to design sampling matrices in MRI.}\label{s:mri}
Compressed sensing has direct applications in MRI \citep{Lust08}. MRI measurements are both costly and time consuming and compressed sensing results allow reconstruction of MRI images using fewer measurements compared to classical signal processing theory. However, MRI hardware imposes specific constraints on sampling which make random sampling patterns highly impractical, and this part is dedicated to the design of efficient sampling patterns for MRI minimizing the sampling matrix's $M^*$.

\subsection{MRI Setting}\label{ss:mri-setting}
We consider a stylized MRI reconstruction problem on 2D images, and follow the setting introduced by \citep[Section 2]{Boye16}. The goal is to reconstruct an image $u : \mathbb{R}^2 \xrightarrow{} \mathbb{C}$, from observations of its 2D Fourier Transform $\hat{u}$ at a finite number of points. This means that only the set $\{\hat{u}(s_i) | s_i \in \mathcal{S}\}$ is known, where $\mathcal{S}$ is the finite set of acquisition points. Classical compressed sensing theory would suggest to draw the elements of $\mathcal{S}$ at random with a certain probability distribution depending on the sensing matrix. However due to physical constraints in the acquisition process, $\mathcal{S}$ cannot be arbitrary. The acquisition points should lay on a continuous admissible path on the plane. 
\clearpage
More precisely, given an acquisition duration $T$ and a sampling period $\Delta t$, $\mathcal{S}$ should be of the form $\{s(i\Delta t), i \in [|1,\lfloor{T}/{\Delta t}\rfloor|] \}$ with the path $s$ in 
\[
\mathcal{S}_{T,A,b,\alpha,\beta} = \{s \in \mathcal{C}^2([0,T],\mathbb{R}^2) \;|\; ||\dot{s}||_{\infty} \leq \alpha , ||\ddot{s}||_{\infty} \leq \beta, A(s) = b\}
\]
where $\alpha$ and $\beta$ are constants imposed by the acquisition process and $A(s) = b$ form additional affine constrains.

In the following, $u$ is a discrete image in $\mathbb{C}^{n \times n}$. We will constantly identify the matrix $X \in \mathbb{C}^{n\times n}$ with the vector $\vect({X}) \in \mathbb{C}^{n^2}$, in particular if $A \in \mathbb{C}^{n^2\times n^2}$ refers to a linear 2D transform, we will consider $AX$ as the matrix formed as $A\vect(X)$. Once $\mathcal{S}$ is determined, there exists a linear mapping $L_{\mathcal{S}}$ from $\mathbb{C}^{n \times n}$ to $\mathbb{C}^{|\mathcal{S}|}$, such that $L_{\mathcal{S}}u = \hat{u}_\mathcal{S}$ where $\hat{u}_\mathcal{S}$ is the vector constituted with Fourier transform of $u$ evaluated at the elements of $\mathcal{S}$.
The problem of reconstructing $u$ is 
\BEQ\label{eq:MRI-rec-noprior}
\BA{ll}
\mbox{find} & x\\
\mbox{subject to} & L_{\mathcal{S}}x=\hat{u}_{\mathcal{S}},
\EA
\EEQ
which has no chance to find $u$ without prior information. However $u$ being a natural image, we can assume $u$ is sparse, or at least compressible in a wavelet basis, i.e. writing $H$ the 2D Discrete Wavelet Transform, one can assume that $Hu$ is compressible. This comes down to the classical compressed sensing formulation, where the recovery problem is written
\BEQ\label{eq:MRI-rec-prior}
\BA{ll}
\mbox{minimize} & \|Hx\|_1\\
\mbox{subject to} & L_\mathcal{S}x=\hat{u}_{\mathcal{S}},
\EA
\EEQ
with respect to $x \in \mathbb{C}^{n^2}$. To retrieve the form of \eqref{eq:l1-dec} the problem can be rewritten 
\BEQ\label{eq:MRI-rec-prior-}
\BA{ll}
\mbox{minimize} & \|x\|_1\\
\mbox{subject to} & L_{\mathcal{S}}H^Tx=\hat{u}_{\mathcal{S}},
\EA
\EEQ
with respect to $x \in \mathbb{C}^{n \times n}$. In the context of MRI the image $u$ is in fact a real signal, which means the minimization is performed on $\reals^{n \times n}$. In addition, the theory in Section~\ref{s:dico} is presented for real dictionaries, so in the following we let
\[
A_{\mathcal{S}}=
  \begin{bmatrix}
    \text{Real}(L_{\mathcal{S}}H^T)\\
    \text{Imag}(L_{\mathcal{S}}H^T)
  \end{bmatrix}.
\]
The cardinal of $\mathcal{S}$ is related to the acquisition time. Given an acquisition time limit that translates into an integer $m$, the goal is to find an admissible acquisition set $\mathcal{S}$ of cardinal $m$ that best recovers the initial image~$u$, solving
\BEQ\label{eq:MRI-rec}
\BA{ll}
\mbox{minimize} & \|x\|_1\\
\mbox{subject to} & A_{\mathcal{S}}x=\begin{bmatrix}
  \text{Real}(\hat{u}_{\mathcal{S}})\\
  \text{Imag}(\hat{u}_{\mathcal{S}})
\end{bmatrix},
\EA
\EEQ
with respect to $x \in \reals^{n \times n}$.
When the elements of $\mathcal{S}$ are taken on a regular Cartesian grid, here $[1,n]^2$, the operator $L_\mathcal{S}$ is a simple 2D Discrete Fourier Transform on computation can be performed efficiently. Writing $F_2 \in \mathbb{C}^{n^2\times n^2}$ the 2D Discrete Fourier Transform matrix, $L_{\mathcal{S}} = F_{2,\mathcal{S}}$ formed using the rows of $F_2$ with index  $i+n(j-1)$ for $(i,j)\in\mathcal{S}$. In the following we identify $[1,n]^2$ and $[1,n^2]$.
Experiments will be presented in this context of Cartesian sampling but can be extended to noncartesian sampling using Nonuniform Fourier Transforms.

\subsection{Selection of the sampling set}
In view of Section~\ref{s:dico}, one can use the quantity $M^*(A_\mathcal{S})$, for a $\mathcal{S}\subset [1,n^2]$, as a proxy for the sparse recovery threshold of the sampling matrix $A_\mathcal{S}$. Thus we will use the $M^*$ to distinguish the sparse recovery power through \eqref{eq:MRI-rec} of different sampling sets $\mathcal{S}$.

In order to deal with tractable problems and to take into account the acquisition constraints mentioned previously we choose to impose some restrictions on sampling sets. It is common in the MRI community to sample the Fourier coefficients along spirals (e.g \citep{Boye16}). Sampling along spirals allows to get a high concentration of sampling points in the center of the image in Fourier space, which correspond to coefficients of low energy. It appears crucial to sample enough low energy Fourier coefficients to obtain well reconstructed images (similar to classical Fourier sampling). 

We consider the problem of choosing $N$ spirals among a set $\mathcal{S}_p$ of $N_s$ predefined spirals. Each spiral in $\mathcal{S}_p$ is obtained as $\left\{\left(\lfloor n\rho\cos(n\theta)\rfloor,\lfloor n\rho\sin(n\theta)\rfloor\right)+c, n\in[0,n_s-1]\right\}$ with $\rho>0,\theta\in\reals,c\in\N^2$ and $n_s\in\N$. We choose $c$ as the center of our 2D images in $\reals^{n\times n}$, $n_s\in \N^*$, and given two vectors $\theta \in \reals^{N_s}$, $\rho \in \reals_+^{N_s}$, we have
\BEQ
\mathcal{S}_p^{\theta,\rho,n_s,c} = \left\{ \left\{\left(\lfloor n\rho_i\cos(n\theta_i)\rfloor,\lfloor n\rho_i\sin(n\theta_i)\rfloor\right)+c, n\in[0,n_s-1]\right\}, i\in[1,N_s] \right\}
\EEQ

Our goal is then to select $N$ spirals among $\mathcal{S}_p^{\theta,\rho,n_s,c}$ to create a sampling set $\mathcal{S}$ with the better reconstruction performances. To do so we we search for the set $\mathcal{S}$ such that $M^*(A_\mathcal{S})$ is the lowest. Finding the minimum of $M^*(A_\mathcal{S})$ for $\mathcal{S}= \bigcup_{i=1}^N \mathcal{S}_i$ with $\mathcal{S}_i\subset\mathcal{S}_p^{\theta,\rho,n_s,c}$ requires $\begin{pmatrix}
N_s\\
N
\end{pmatrix} M^*$ computations which will be prohibitive. Thus we adopt a greedy strategy to select spirals with the lowest $M^*$. This strategy is describe in Algorithm~\ref{algo:greedy-m*} and requires $O(NN_s)$ $M^*$ computations.
\begin{algorithm}[!ht]
\caption{Greedy $M^*$}
\label{algo:greedy-m*}
\begin{algorithmic}
\STATE \algorithmicrequire\; $A$ linear operator on $\reals^{n^2}$, $\mathcal{S}_p^{\theta,\rho,n_s,c}=\bigcup_{i=1}^{N_s}\mathcal{S}_p^i$ a set of $N_s$ spirals in $\reals^{n\times n}$, $N\in\N^*$ the number of spirals to select. 
\STATE $\mathcal{S}_0 = \{\emptyset\}$.
\FOR{ $k$ from 1 to $N$}
\STATE $\mathcal{S}_k = \underset{\substack{l \in [1,N_s],\\ \mathcal{S} = \mathcal{S}_{k-1}\cup \mathcal{S}_p^{l}}}{\argmin }M^*(A_\mathcal{S}) $ where $A_\mathcal{S}$ is composed of the rows of $A$ with indexes $i+n(j-1)$ for $(i,j)\in\mathcal{S}$,
\ENDFOR
\STATE \algorithmicensure\; $S_{N}$.
\end{algorithmic}
\end{algorithm}

\subsection{Computing M* efficiently}\label{ss:m*-comp} In the dictionary learning applications, the size of the matrices involved was small and stochastic gradient descent steps were cheap because each step relied on solving only a small linear program. However in Algorithm~\ref{algo:greedy-m*}, adding a spirals to the current sampling set requires the computation of $N_s-i-1$ $M^*$ at iteration $i$. Recall from \eqref{eq:m*-fun} that $M^*$ is computed through a sampling strategy, involving the resolution of a substantial number of these linear programs (typically 500). For a matrix $A \in \reals^{m\times n}$ and $g \sim \mathcal{N}(0,I)$, the linear program to be sampled has the following form
\BEQ\label{eq:m*-eq_bis}
\BA{ll}
\mbox{minimize} & \|x+Fg\|_{\infty}\\
\mbox{subject to} &  F^Tx = 0\\
\EA
\EEQ
with respect to $x \in \reals^{n}$, writing $F$ basis of the nullspace of $A$, such that $F^TF = I$. The program contains $n+1$ variables and at least $3n-m$ constrains. Given that $F^Tx = 0 \iff x = A^Ty$ with $y\in\reals^{m}$ the problem becomes 
\BEQ\label{eq:m*-eq-less-var}
\BA{ll}
\mbox{minimize} & \|A^Tx+Fg\|_{\infty}\\
\EA
\EEQ
with respect to $x \in \reals^{m}$. This is a linear program with $m+1$ variables and $2n$ constrains which is a nice improvement since typically $m << n$. This problem is a $L_{\infty}$ norm linear regression, and is called discrete linear Chebyshev approximation in the literature. Several algorithms have been developed to solve these particular linear programs efficiently, with \citet{Cadz02} for example giving a simplex-like method, or \citet{Cole92} proposing an interior-point method that can get past the nondifferentiability plane through specific line-searches. Actually this task is done very efficiently for relatively small $n$ ans $m$ by modern interior-point solvers such as MOSEK but it becomes quickly computationally prohibitive for larger dimensions.

The problem under its current form starts becoming intractable on common machines when considering images of size $128\times128$  which is still very far from real MRI images, and this becomes a real issue when tackling small $256\times 256$ images, because the full sensing matrix does not fit in memory. However, the sensing operator is a composition of Fourier and Wavelets transforms that can be computed very efficiently, but classical LP solvers require a full dense matrix as input. We can fix this issue using solvers that only access this matrix as a linear operator. To do so, we implemented an operator version of the indirect Splitting Cone Algorithm introduced in \citep{ODon16}. SCS is a first order method to solve the linear conic problem
\BEQ\label{eq:scs-cone}
\BA{ll}
\mbox{minimize} & c^Tx\\
\mbox{subject to} &  Ax \preceq_{\mathcal{K}} b\\
\EA
\EEQ
with respect to $x \in \reals^n$, where $A \in \reals^{m \times n}, c \in \reals^n, b \in \reals^{m}$ and $\mathcal{K}$ is a non-empty convex cone. When $\mathcal{K}$ is simply the positive orthant $\reals^m_+$, this corresponds to a classical linear program. The algorithm relies on projection on $\mathcal{K}$ which is really easy in the case of the positive orthant and for projection on a linear subspace. These last projection comes down to solve linear systems involving the constrain matrix $A$. These systems can be solved iteratively using the conjugate gradient method which only performs matrix vector multiplication, allowing the use of abstract linear operators as constrains, hence never forming the full dense matrix $A$.

Problem \eqref{eq:m*-eq-less-var} can be reformulated to match the form of \eqref{eq:scs-cone} as 
\BEQ\label{eq:LP-m*}
\BA{ll}
\mbox{minimize} & t\\
\mbox{subject to} &  \begin{bmatrix}
  A & \begin{array}{c}
       -1 \\
       \vdots
  \end{array}\\
  -A & \begin{array}{c}
       \vdots\\
       -1
  \end{array}
\end{bmatrix} \begin{bmatrix}
  x\\
  t
\end{bmatrix} = \begin{bmatrix}
  -Fg\\
  Fg
\end{bmatrix}\\
\EA
\EEQ
where the minimization is performed with respect to the variables $x\in \reals^n$ and $t\in \reals$.

\subsection{Preliminary Numerical Results for Spirals Selection}\label{ss:num-mri}

Tests of Algorithm~\ref{algo:greedy-m*} are performed on operator $A\in\reals^{128^2\times128^2}$ which is the composition of the 2D Fourier operator and the 2D Haar wavelet operator. $\mathcal{S}_p^{\theta,\rho,n_s,c}$ is constituted of $N_s = 48$ discretized spirals containing $n_s = 300$ points. Some elements of $\mathcal{S}_p^{\theta,\rho,n_s,c}$ are displayed in Figure~\ref{fig:spirals}.
\begin{figure}[!ht]
    \centering
    \begin{tabular}{cc}
    \includegraphics[width=0.3\linewidth]{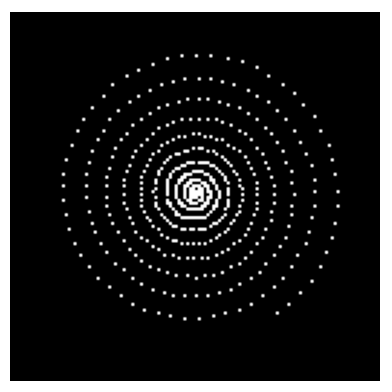} & \includegraphics[width=0.3\linewidth]{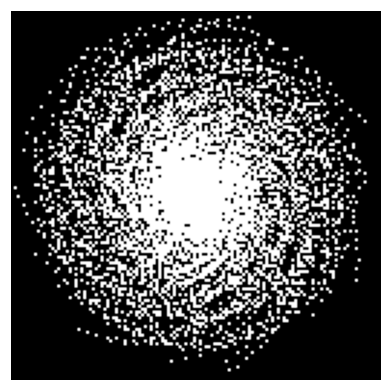} \\
    \end{tabular}
    
    \caption{Left : an element of $\mathcal{S}_p^{\theta,\rho,n_s,c}$. Right : the set $\mathcal{S}_*$ obtained by Algorithm~\ref{algo:greedy-m*} with the setting described in Section~\ref{ss:num-mri}. }
    \label{fig:spirals}
\end{figure}
We denote by $\mathcal{S}_*$ the output of Algorithm~\ref{algo:greedy-m*} after greedily choosing $N=16$ spirals in $\mathcal{S}_p^{\theta,\rho,n_s,c}$ according to their $M^*$ values. We compare the quality of the reconstruction of three images (displayed on Figure~\ref{fig:rec-spirals}) by solving \eqref{eq:MRI-rec} for $\mathcal{S}_*$, and sampling sets constituted of $N$ to $N+4$ spirals selected randomly in $\mathcal{S}_p^{\theta,\rho,n_s,c}$. Since \eqref{eq:MRI-rec} is a linear program we also solve it using the SCS solver from Section \ref{ss:m*-comp}. In addition, as the different spirals may contain common points, the sampling sets cardinals appear as appropriate measures of the sampling complexity and should be correlated with reconstruction quality. Another natural competitor to $\mathcal{S}_*$ would be $\mathcal{S}_{\text{max}}$ the set of $N$ spirals with maximal cardinals, however it is too difficult to compute and we consider instead the set of $N$ spirals obtained by greedily choosing spirals that make the cardinal the bigger (same as Algorithm~\ref{algo:greedy-m*} but using $\Card(\mathcal{S})$ instead of $M^*(A_\mathcal{S})$).
\begin{figure}[!ht]
    \centering
    \begin{tabular}{ccc}
    \includegraphics[width=0.26\linewidth]{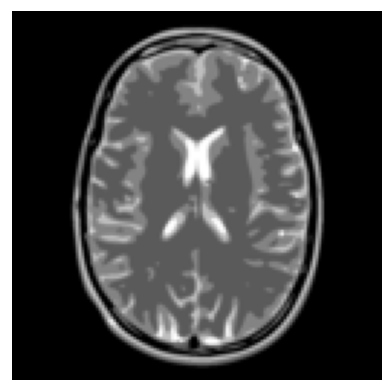}& \includegraphics[width=0.26\linewidth]{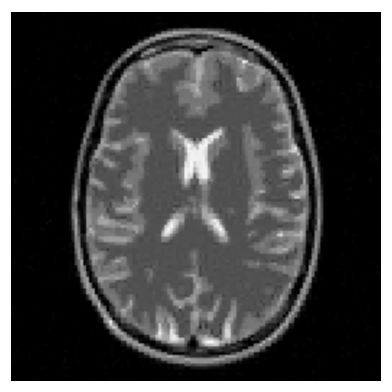} &\includegraphics[width=0.38\linewidth]{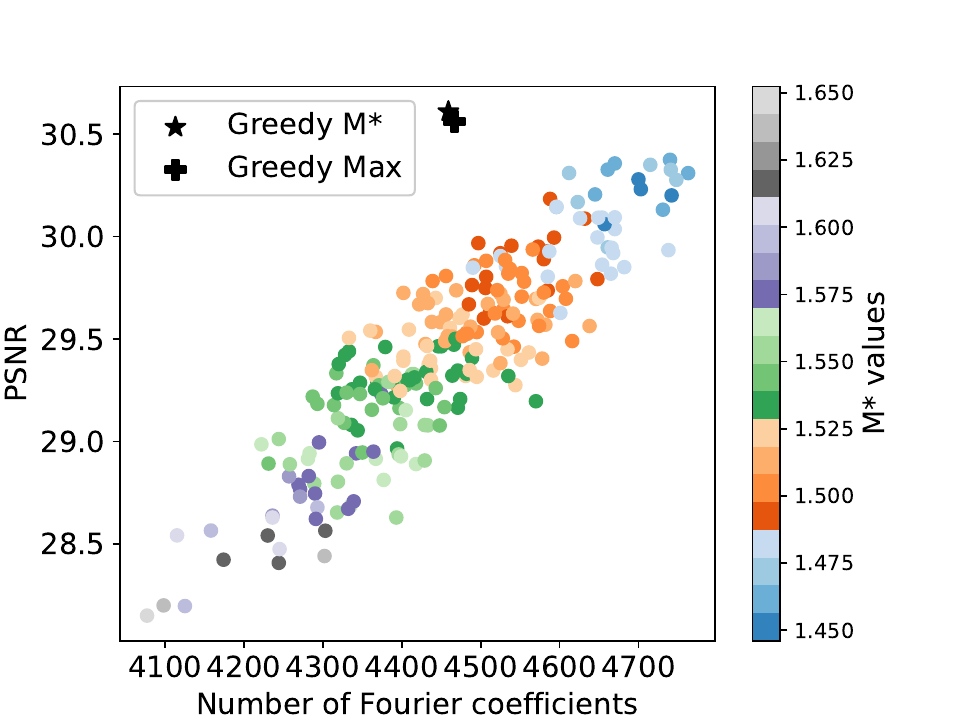} \\
    \includegraphics[width=0.26\linewidth]{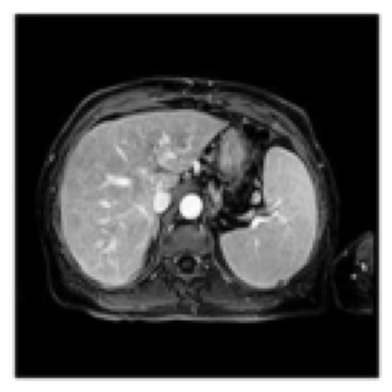}  & \includegraphics[width=0.26\linewidth]{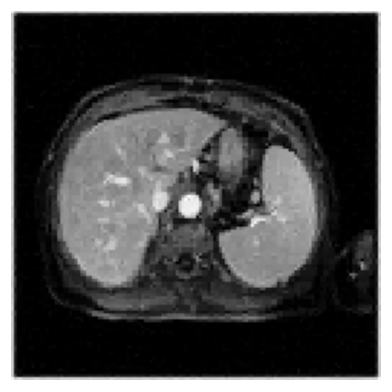} & \includegraphics[width=0.38\linewidth]{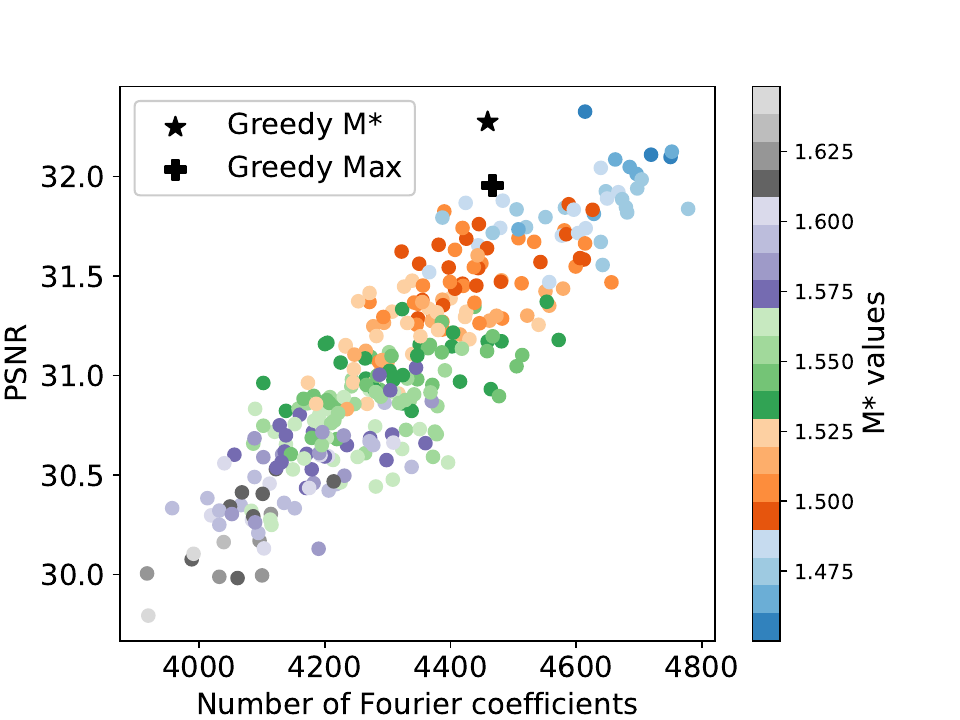} \\
    \includegraphics[width=0.26\linewidth]{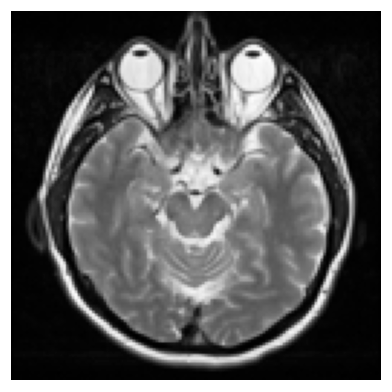}& \includegraphics[width=0.26\linewidth]{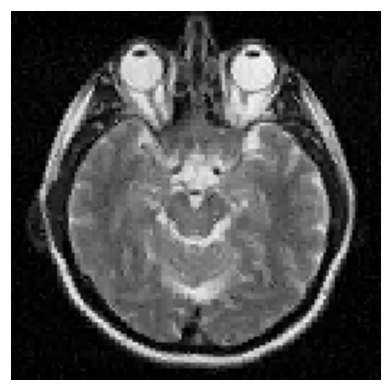}&\includegraphics[width=0.38\linewidth]{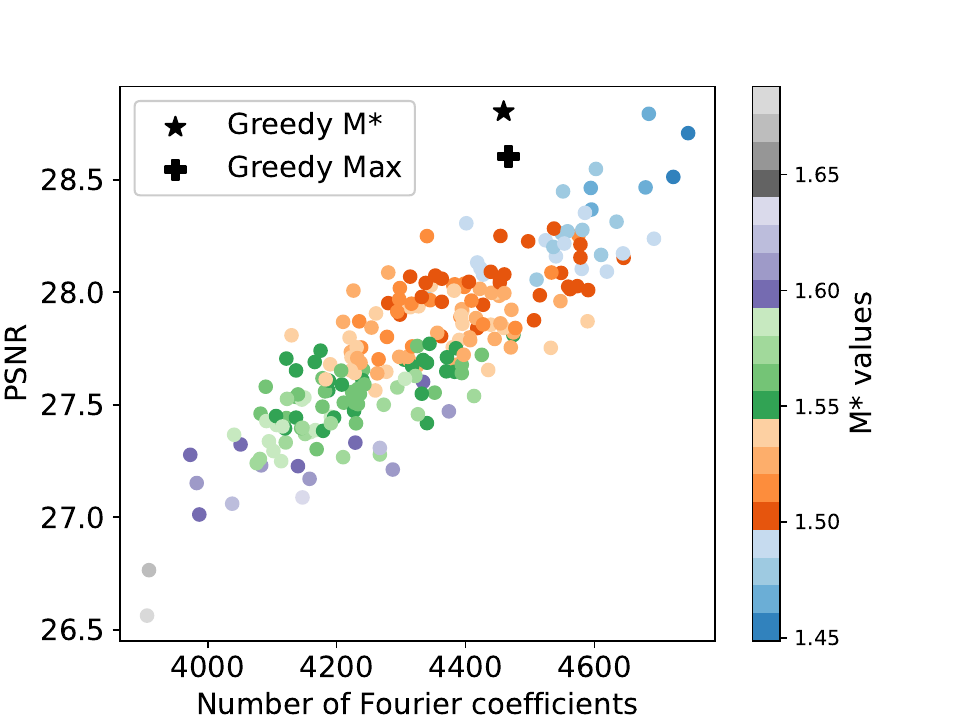} 
    \end{tabular}
    
    \caption{Reconstruction of MRI image through \eqref{eq:MRI-rec} for different sampling sets $\mathcal{S}$. Left : Original Images. Middle : Reconstructions of the left images obtained after solving \eqref{eq:MRI-rec} using our matrix free SCS implementation with precision $1.5\;10^{-3}$ with sampling set $\mathcal{S}_*$ detailed in Section~\ref{ss:num-mri}. Right : Each color dot corresponds to a reconstructed image obtained after solving \eqref{eq:MRI-rec} using our matrix free SCS implementation with precision $1.5\;10^{-3}$ with sampling set $\mathcal{S}$ constituted of spirals selected randomly in $\mathcal{S}_p^{\theta,\rho,n_s,c}$. Y-axis is the PSNR of the reconstructed images, x-axis is $\Card(\mathcal{S})$. The colors correspond to the value of $M^*(A_\mathcal{S})$. Greedy $M^*$ corresponds to $\mathcal{S}_*$ and Greedy Max to $\mathcal{S}_{\max}$ described in Section~\ref{ss:num-mri}. }
    \label{fig:rec-spirals}
\end{figure}

One can make several observations on Figure~\ref{fig:rec-spirals} (Right). First we observe that as we expect, the number of Fourier coefficients that is sampled is positively correlated with the quality of the reconstruction through \eqref{eq:MRI-rec}. Secondly we can notice that given a fixed number of Fourier coefficients in the sampling sets, reconstruction qualities appeared to be partially ordered by decreasing $M^*$ values of the sampling matrices. This follows the motivation of Section~\ref{s:dico},  {\em having a lower M* leads to better reconstruction properties}. Finally we observe that the sampling set $\mathcal{S}_*$ obtained with Algorithm~\ref{algo:greedy-m*} seems to lead to better reconstruction that other sampling sets with comparable numbers of Fourier coefficients.

\section{Conclusion and Future Work} We derived a tractable proxy for the sparse recovery threshold sensing matrices (that were called dictionary matrices or sampling matrices depending on the application). We aimed at exhibiting strong correlations between high values of sparse recovery threshold (that we enforced by maximizing our $M^*$ based lower bound) and high reconstruction (or generalization) power for sensing matrices. We did so in two applications, dictionary learning and compressed sensing MRI. We provide an algorithm to learn dictionaries with low $M^*$ values, and exhibit improved reconstruction quality in inpainting problems. Then we proposed an $M^*$ based greedy algorithm to design sampling sets for MRI acquisition and exhibits strong reconstruction power of such sets as well as clear link between $M^*$ values and reconstruction power.


\section*{Acknowledgements} 
A.A. is at the d\'epartement d'informatique de l'ENS, l'\'Ecole normale sup\'erieure, UMR CNRS 8548, PSL Research University, 75005 Paris, France, and INRIA. AA would like to acknowledge support from the {\em ML and Optimisation} joint research initiative with the {\em fonds AXA pour la recherche} and Kamet Ventures, a Google focused award, as well as funding by the French government under management of Agence Nationale de la Recherche as part of the "Investissements d'avenir" program, reference ANR-19-P3IA-0001 (PRAIRIE 3IA Institute).

\clearpage
\bibliography{mybiblio,MainPerso}
\bibliographystyle{plainnat}

\section{Supplementary Material}

 \subsection{Compression Experiments}\label{ss:comp-exp-sup}
 This is the setting of Section~\ref{sss:classic-dict}. The number of atoms in the dictionaries has been set to $p = 4 n = 256$. The training set is formed by $200p = 51200$ patches selected randomly in four training images. Both KSVD and penalized dictionary methods are applied for $150$ iterations, for different sparsity levels $S$ between $2$ and $10$.
 
  The dictionary obtained by $M^*$ penalization is not as sharp as the dictionary learned with KSVD, yet has a lot more structure than random Gaussian dictionaries (Figure~\ref{fig:dicos-comp}).

\begin{figure}[!ht]
\begin{minipage}[c]{0.4\linewidth}
\centering
\begin{tabular}{c}
        \hspace{-0.5cm}\includegraphics[scale=0.2]{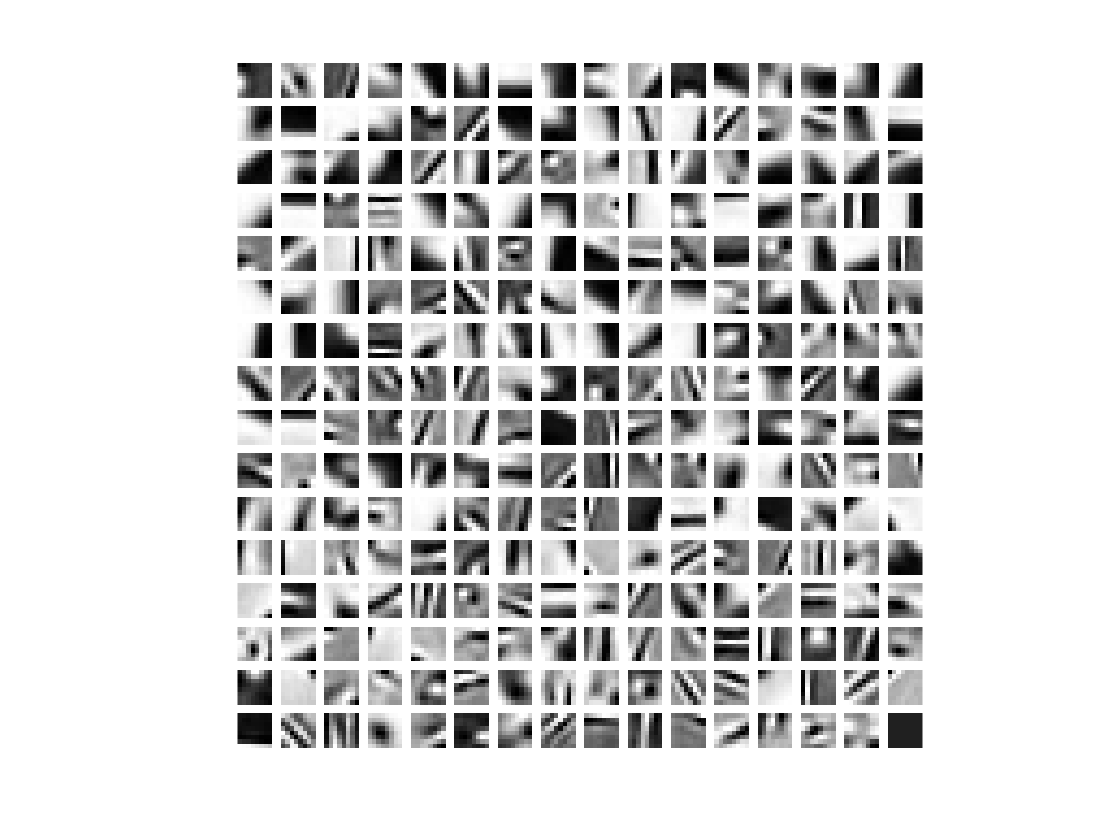}\\ KSVD \\
        \hspace{-0.5cm}\includegraphics[scale=0.2]{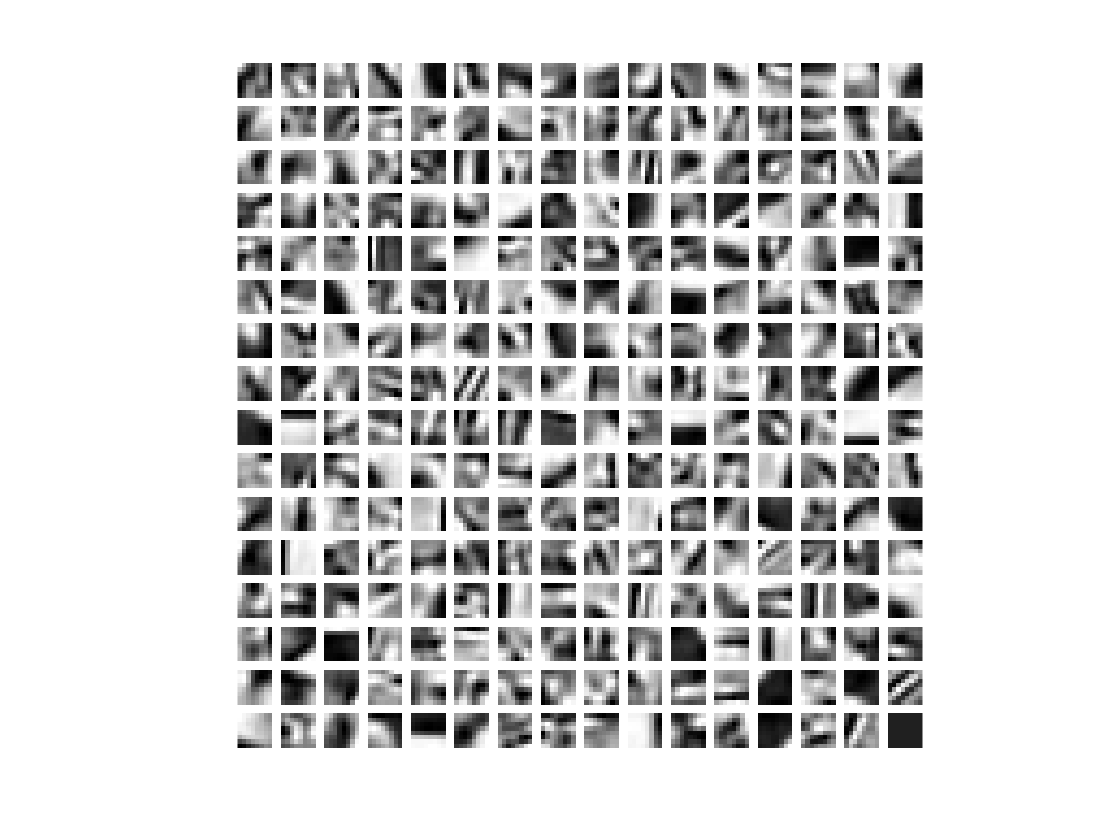}\\ $M^*$
    \end{tabular}
    \caption{Example of dictionaries learned by both KSVD and M*-Regularization, with a training sparsity $S = 5$ and a regularization parameter $\mu = 10^8$. (Gaussian $M^*$ for this dimensions is $1.517 \pm 0.003$). Left: KSVD, $M^* = 1.686 \pm  0.004$. Right: $M^*$ penalization, $M^* = 1.558 \pm 0.003$.}
    \label{fig:dicos-comp}
\end{minipage}
\hfill
\begin{minipage}[c]{0.4\linewidth}
\centering
\vspace{-1.2cm}
\begin{tabular}{c}
        \includegraphics[width=0.85\linewidth,height=0.85\linewidth]{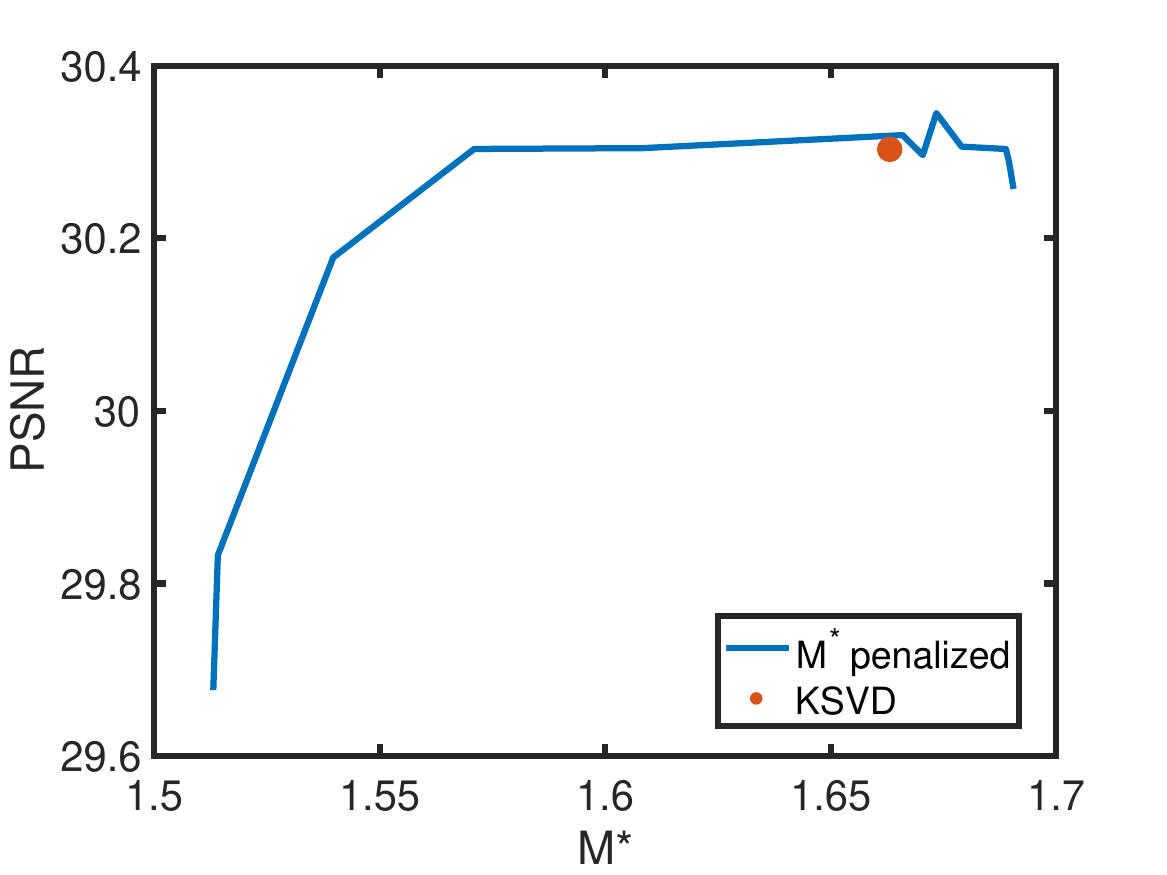}\vspace{0.5cm}\\
        \includegraphics[width=0.85\linewidth,height=0.85\linewidth]{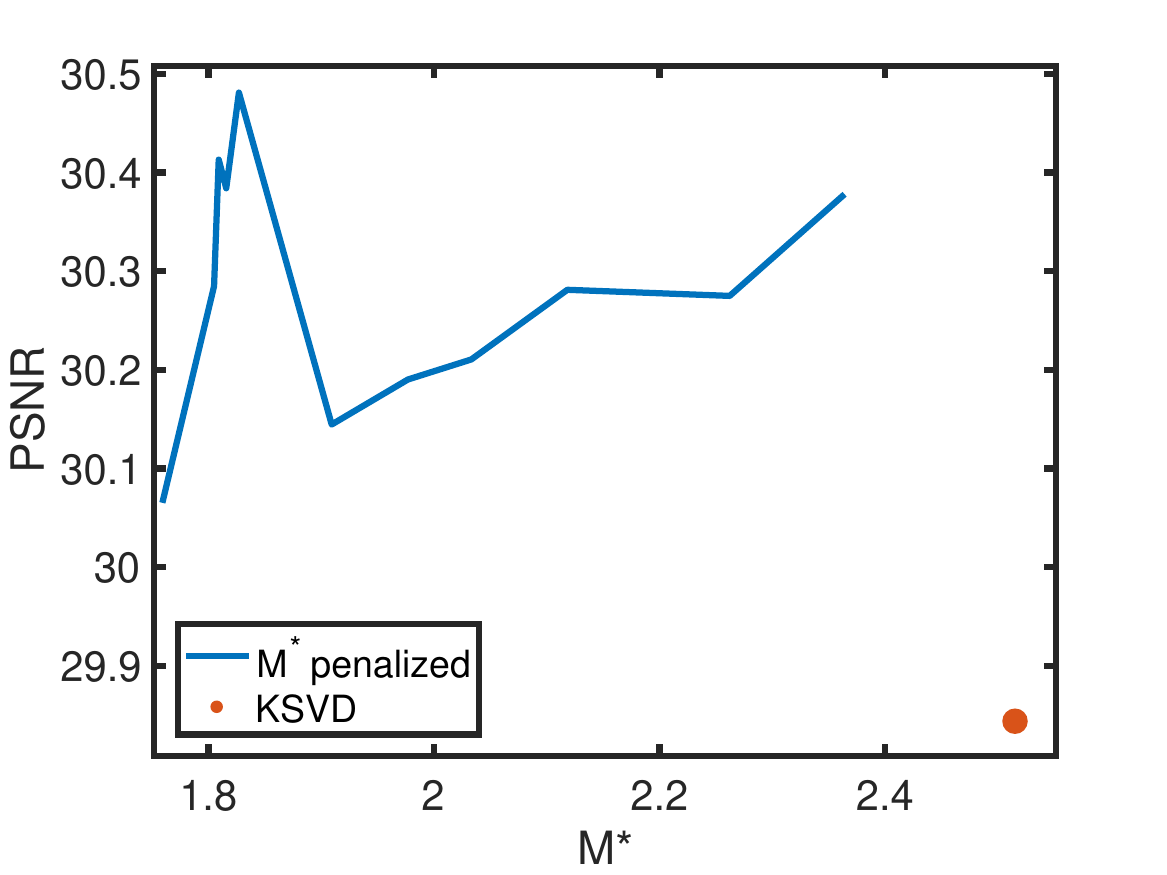}
    \end{tabular}
    \vspace{0.75cm}
    \caption{Training PSNR and $M^*$ Evolution with the regularization parameter $\mu$ with $S = 4$. Left: Gaussian initialization. Right: deterministic initialization. }
    \label{fig:deter-init}
\end{minipage}%
\end{figure}

To measure the quality of a dictionary $D$ in the compression setting, we formed a test set of 21 standard gray scale $512 \times 512$ images. Each image is decomposed in non overlaying patches. This means for instance that a $512\times512$ image is cut into $64\cdot64 = 4096$ adjacent patches of size $8 \times 8$. Each image is then represented as a set of patches $Y \in \reals^{64\times 4096}$ and is approximated by $DX$ where $X$ is obtained like \eqref{eq:rep-fun} with a given reconstruction sparsity $k$ (which is not necessarily the same as the training sparsity $S$). This corresponds to the compression factor: the smaller $k$, the more compressed the images are, so compression is measured by the cardinality of the representations of the images in the dictionary. 

For a given set of non overlaying patches $Y \in \reals^{64\times 4096}$ representing an image and a cardinality $k$, we define
\BEQ\label{eq:rep-fun}
\BA{rll}
Y_k(D) &\triangleq DX, & \\
X &= \mbox{argmin.} & ||Y-DX||^2_F\\
& \mbox{s.t.} & ||X_j||_0 \leq k,\quad\mbox{$j=1,\ldots,4096$}
\EA
\EEQ
where the minimization is performed with respect to $X\in\reals^{256\times4096}$. Here, $Y_k(D)$ corresponds to the matrix where each column is an approximation of the corresponding column of $Y$ using a linear combination of $k$ atoms of $D$.

\begin{figure*}[!ht]
    \centering
    \begin{tabular}{ccc}
        \includegraphics[scale=0.27]{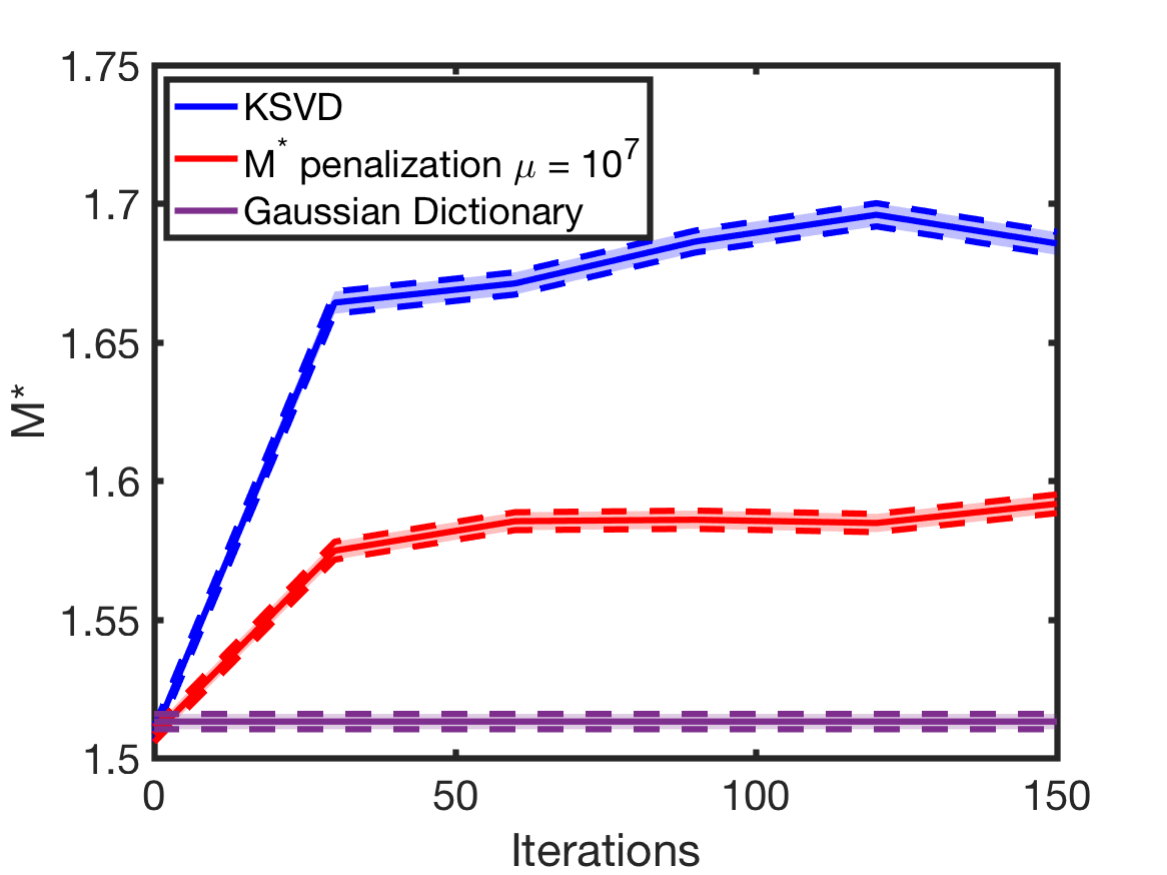}&
        \includegraphics[scale=0.27]{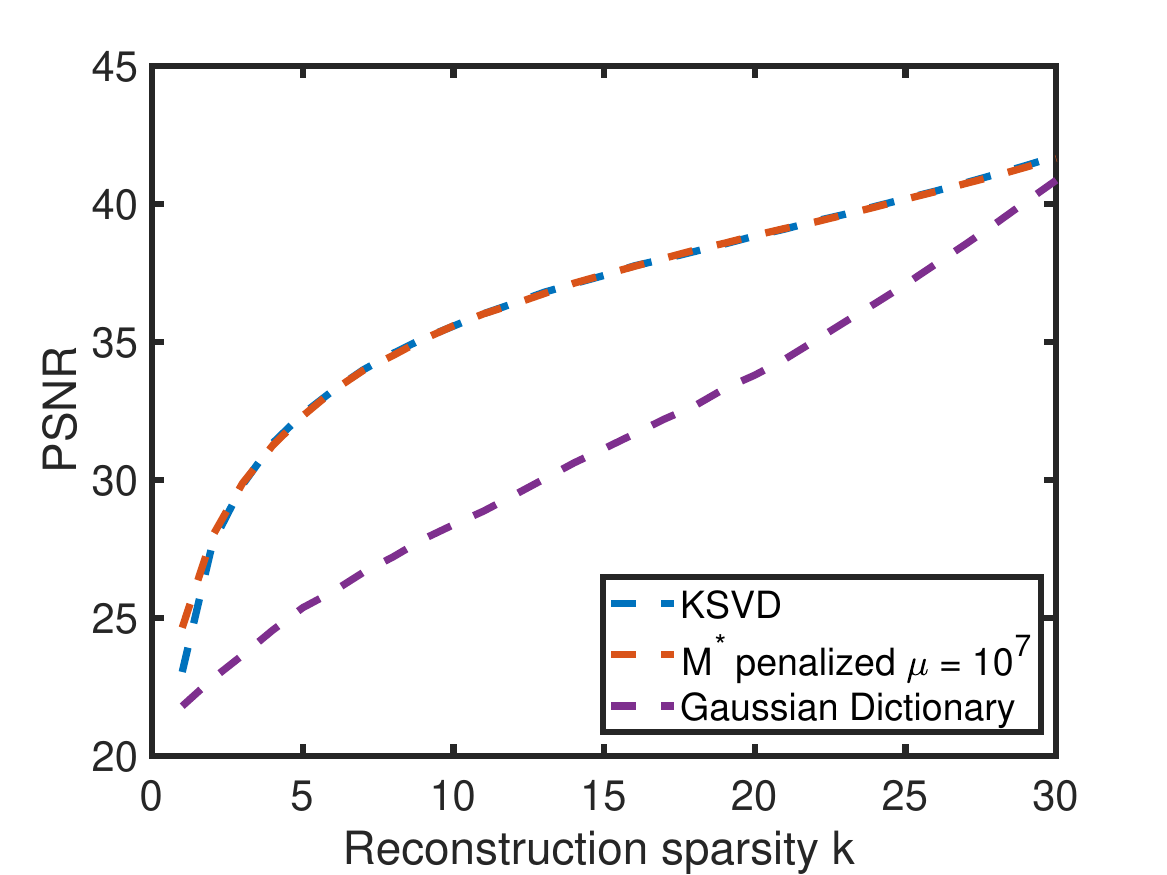}&
        \includegraphics[scale=0.27]{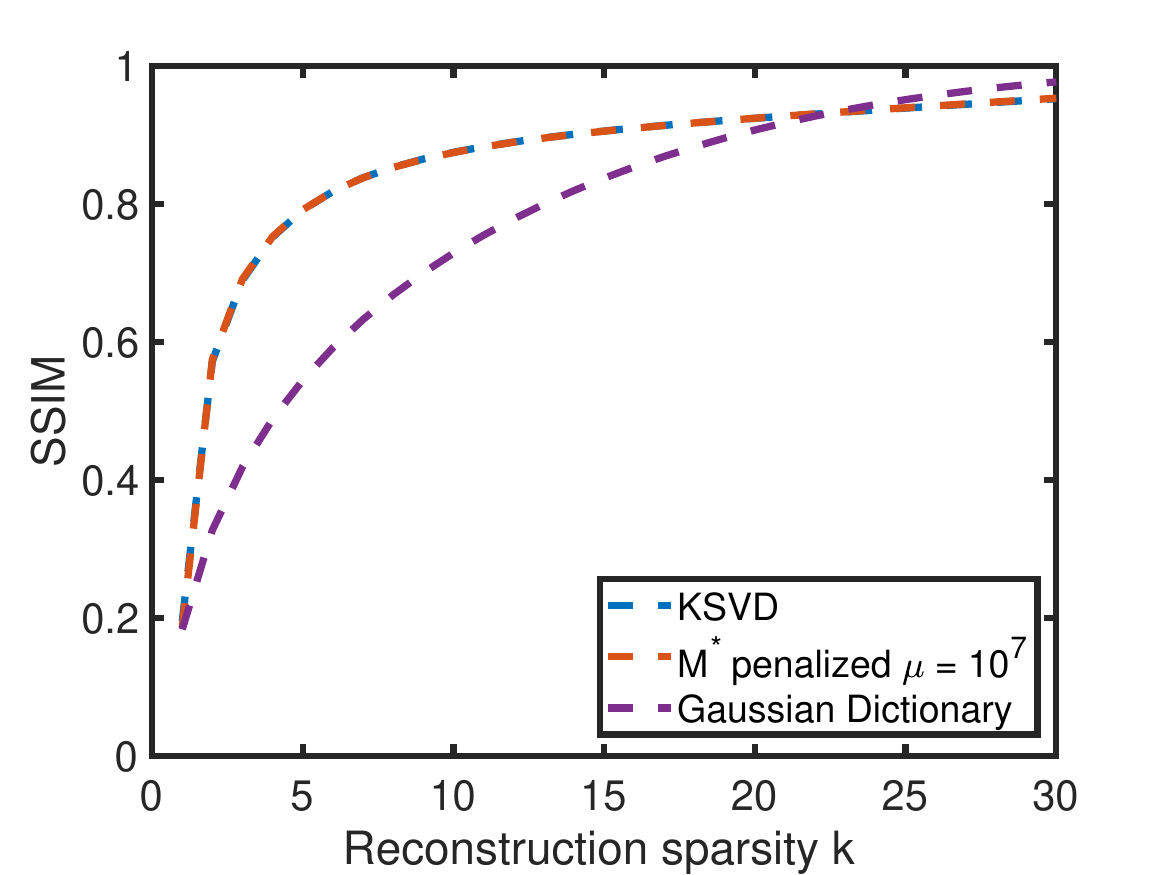}\\
        \includegraphics[scale=0.27]{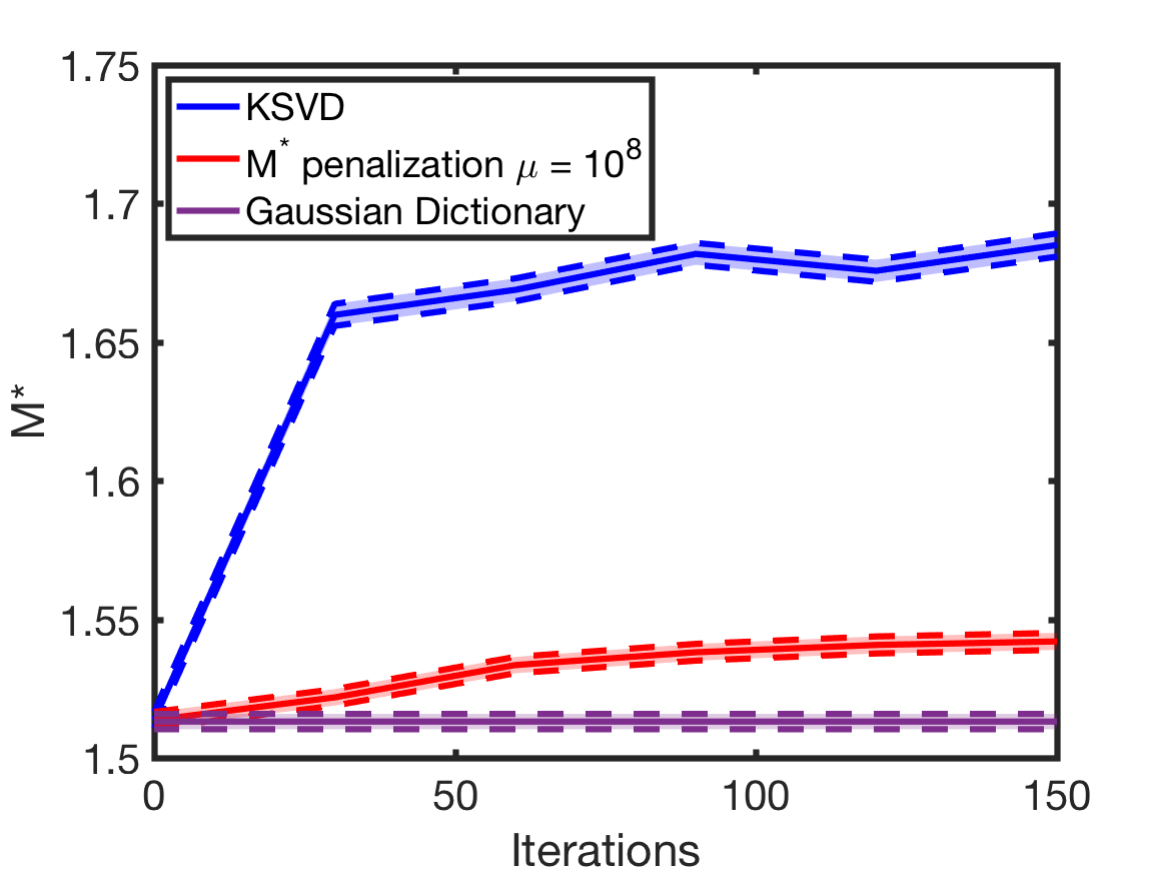}&\includegraphics[scale=0.27]{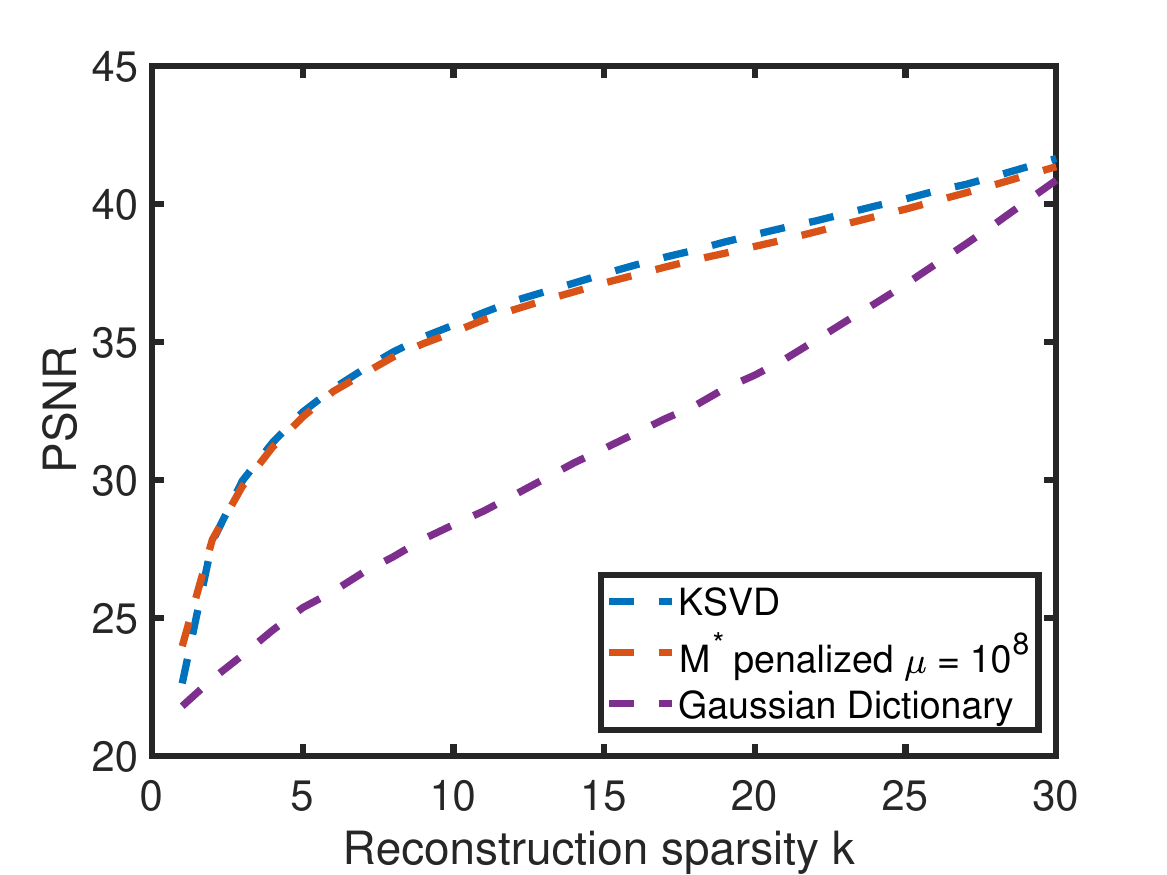}
        &\includegraphics[scale=0.27]{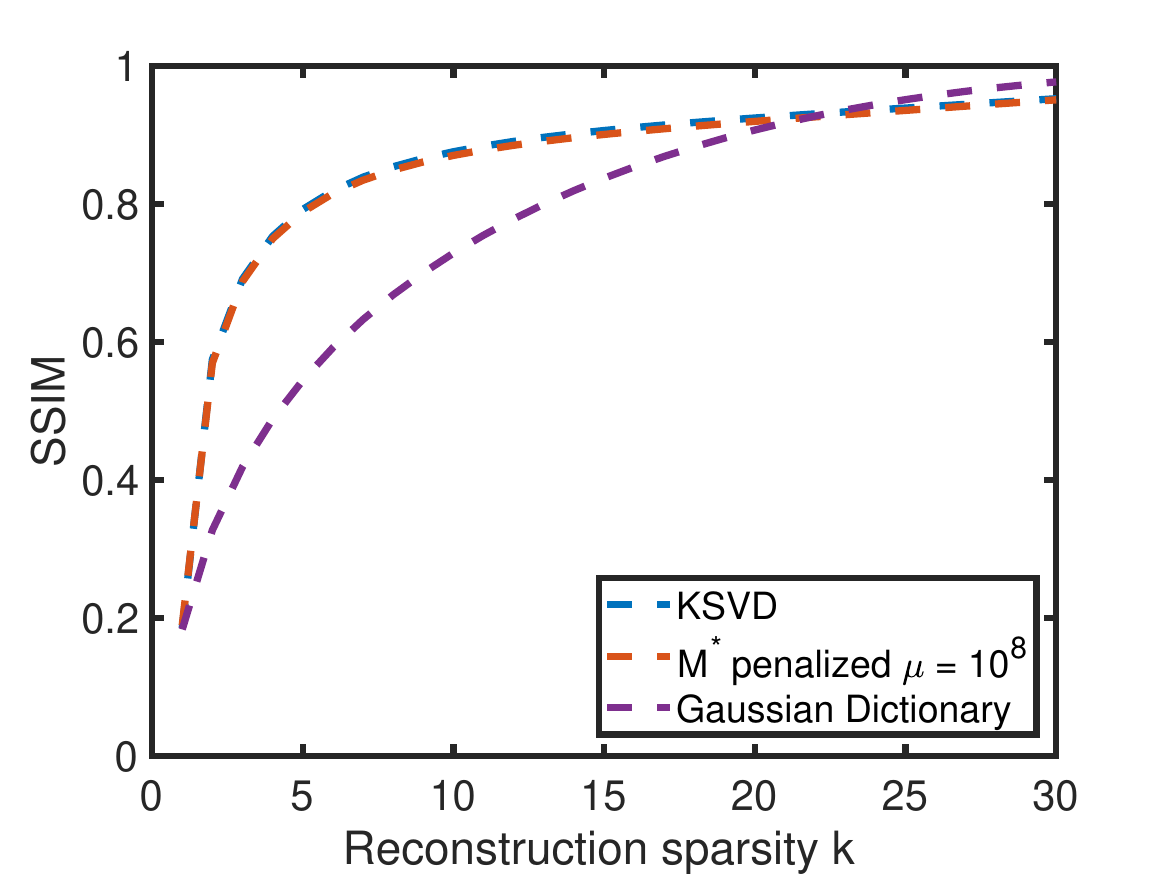}
    \end{tabular}
    \caption{Compression experiments with training sparsity $S = 5$. Top: $\mu=10^7$. Bottom: $\mu=10^8$. Left: Evolution of the $M^*$ across algorithm iterations. Middle: PSNR of the compressed test images. Right: SSIM of the compressed test images. }
    \label{fig:dicos-comp-res}
\end{figure*}

We write $D_S$ the dictionary obtained by KSVD with a training sparsity $S$, and $D_S^{\mu}$ the one obtained by the $M^*$ penalized algorithm with sparsity $S$ and regularization parameter $\mu$ in problem~\eqref{eq:dico-m*}. For a patch representation $Y$ of a test image, a reconstruction sparsity $k$ and a penalization coefficient $\mu$, approximation quality for the KSVD (resp. $M^*$ penalized) algorithm is obtained by computing both PSNR and SSIM between the ground truth $Y$, and $Y_k(D_S)$ (resp. $Y_k(D_S^{\mu})$). SSIM is a measurement of structural similarity designed to describe the perceived quality of an image more faithfully than PSNR, which is a pixel to pixel measurement \citep{Wang04a}. In order to plot the aggregate curve in Figure~\ref{fig:dicos-comp-res} on the right, we took the average of the PSNR and SSIM values over all 21 images in the test set, for a range of reconstruction sparsity $k$ between $2$ and $30$.

When using small penalization $\mu$ in problem~\eqref{eq:dico-m*}, the two methods had similar compression performances on the test set, with a minor advantage for $M^*$ penalized dictionary for small values of $k$. In this case, the $M^*$ of the penalized dictionary has an intermediate value between that of the dictionary from KSVD and that of a Gaussian dictionary. Increasing the penalization parameter allows to learn dictionaries with $M^*$ almost as low as Gaussian ones, however these new dictionaries with low $M^*$ don't fit the training data as well and the test PSNR and SSIM become worse than those of KSVD.

Convergence of the algorithms with deterministic initialization has also been experimented, with $D_0$ a constant matrix with columns of norm 1. The KSVD algorithm converges very slowly in this case (if at all). All the columns of the dictionary remain very close to the initial ones. The penalized algorithm on the other hand achieves similar performances on train and test errors compared to the random initialization setting. It also appears that a stronger $M^*$ regularization is suitable to reach better training error, compared with random initialization (see Figure~\ref{fig:deter-init}).

For the compression task, it thus appears that dictionaries learned by KSVD are nearly optimal, if randomly initialized, in the sense that learning a dictionary with lower $M^*$ through $M^*$ penalization method does not improve performance. Regularization does improve convergence when starting from a deterministic matrix.

\end{document}